\documentclass[a4paper]{amsart}
\usepackage[active]{srcltx}
\usepackage[all]{xy}
\usepackage{subfig}
\usepackage{hyperref}
\usepackage{mathrsfs}
\usepackage{esint}
\usepackage{amsmath}
\usepackage{amssymb,latexsym}
\usepackage{mathrsfs}
\usepackage{graphics}
\usepackage{latexsym}
\usepackage{psfrag}
\usepackage{import}
\usepackage{verbatim}
\usepackage{graphicx}
\usepackage[usenames]{color}
\usepackage{pifont,marvosym}
\usepackage[normalem]{ulem}

\theoremstyle{plain}
\newtheorem{lemma}{Lemma}[section]
\newtheorem{theorem}[lemma]{Theorem}
\newtheorem{proposition}[lemma]{Proposition}
\newtheorem{corollary}[lemma]{Corollary}

\theoremstyle{definition}

\newtheorem{definition}[lemma]{Definition}
\newtheorem{remark}[lemma]{Remark}

\numberwithin{equation}{section}

\newcommand{\dom}{\textrm{Dom\,}}

\newcommand{\R}{\mathbb{R}}
\newcommand{\N}{\mathbb{N}}

\newcommand{\supp}{\text{\rm supp}}

\newcommand{\Lip}{\mathrm{Lip}}
\newcommand{\gr}{\textrm{graph}}

\newcommand{\diam}{{\rm diam\,}}

\newcommand{\ve}{\varepsilon}

\newcommand{\cI}{\mathcal{I}}

\newcommand{\f}{\varphi}

\newcommand{\T}{\mathcal{T}}
\newcommand{\I}{\mathcal{I}}

\renewcommand{\L}{\mathcal{L}}
\newcommand{\RCD}{\mathsf{RCD}}

\newcommand{\CD}{\mathsf{CD}}
\newcommand{\Geo}{{\rm Geo}}
\newcommand{\MCP}{\mathsf{MCP}}
\newcommand{\mm}{\mathfrak m}
\newcommand{\qq}{\mathfrak q}
\newcommand{\QQ}{\mathfrak Q}

\newcommand{\sfd}{\mathsf d}

\newcommand{\PP}{\mathsf{P}}
\newcommand{\Opt}{\mathrm{OptGeo}}
\newcommand{\vol}{\mathrm{vol}}

\begin{document}

\title{Quantitative isoperimetry \`a  la  Levy-Gromov}
\author{F. Cavalletti} \thanks{F. Cavalletti: SISSA, Trieste, email: cavallet@sissa.it}
\author{F. Maggi}  \thanks{F. Maggi: University of Texas at Austin, email: maggi@math.utexas.edu}
\author{A. Mondino}  \thanks{A.  Mondino: University of Warwick,  Mathematics Institut,  email: A.Mondino@warwick.ac.uk}


\bibliographystyle{plain}

\begin{abstract} On a Riemannian manifold with a positive lower bound on the Ricci tensor, the distance of isoperimetric sets from geodesic balls is quantitatively controlled in terms of the gap between the isoperimetric profile of the manifold and that of a round sphere of suitable radius. The deficit between the diameters of the manifold and of the corresponding sphere is bounded likewise. These results are actually obtained in the more general context of (possibly non-smooth) metric measure spaces
with curvature-dimension conditions through a quantitative analysis of the transport-rays decompositions obtained by the localization method.
\end{abstract}

\maketitle

\section{Introduction} Comparison theorems are an important part of Riemannian Geometry \cite{Chav06, CE75, Pet}. The typical result asserts that a complete Riemannian manifold with a pointwise curvature bound retains some geometric properties of the corresponding simply connected model space. We are interested here in the {\it Levy-Gromov comparison Theorem}, stating that, under a positive lower bound on the Ricci tensor, the isoperimetric profile of the manifold is bounded from below by the isoperimetric profile of the sphere. More precisely, define the isoperimetric profile of a smooth Riemannian manifold $(M,g)$ by
\[
\cI_{(M,g)}(v)=\inf\left\{\frac{\PP(E)}{\vol_g(M)}:\frac{\vol_g(E)}{\vol_g(M)}=v\right\}\qquad 0<v<1\,,
\]
where $\PP(E)$ denotes the perimeter of a region $E\subset M$. The Levy-Gromov comparison Theorem states that, if ${\rm Ric}_g\ge (N-1)g$, where $N$ is the dimension of $(M,g)$, then
\begin{equation}
  \label{levygromov inequality}
  \cI_{(M,g)}(v)\ge \cI_{(\mathbb{S}^N,g_{\mathbb{S}^N})}(v)\qquad\forall v\in(0,1)\,,
\end{equation}
where $g_{\mathbb{S}^N}$ is the round metric on $\mathbb{S}^N$ with unit sectional curvature; moreover, if equality holds in \eqref{levygromov inequality} for some $v\in(0,1)$, then $(M,g)$ is isometric to $(\mathbb{S}^N,g_{\mathbb{S}^N})$.

Our main result is a quantitative estimate, in terms of the gap in the Levy-Gromov inequality, on the shape of isoperimetric sets in $(M,g)$. We show that isoperimetric sets are close to geodesic balls. Since the classes of isoperimetric sets and geodesic balls {\it coincide} in the model space $(\mathbb{S}^N,g_{\mathbb{S}^N})$, one can see our main result as a {\it quantitative comparison theorem}.   In detail, we show that if ${\rm Ric}_g\ge (N-1)g$ and $E\subset M$ is an isoperimetric set in $M$ with $\vol_g(E)=v\,\vol_g(M)$, then there exists $x\in M$ such that
\begin{equation}
  \label{main theorem intro}
  \frac{\vol_g\big(E\Delta B_{r_N(v)}(x)\big)}{\vol_g(M)}\le C(N,v)\,\Big(\cI_{(M,g)}(v)-\cI_{(\mathbb{S}^N,g_{\mathbb{S}^N})}(v)\Big)^{{\rm O}(1/N)}
\end{equation}
where $B_r(x)$ denotes the geodesic ball in $(M,g)$ with radius $r$ and center $x$, $r_N(v)$ is the radius of a geodesic ball in $\mathbb{S}^N$ with volume $v\,\vol_{g_{\mathbb{S}^N}}(\mathbb{S}^N)$ and $\cdot \Delta \cdot$ denotes the symmetric difference of sets. More generally the same conclusion holds for every $E\subset M$ with $\vol_g(E)=v\,\vol_g(M)$, provided $\cI_{(M,g)}(v)$ on the right-hand side of \eqref{main theorem intro} is replaced by $\PP(E)/\vol_g(M)$. 
\\

We approach the proof of \eqref{main theorem intro} from the synthetic point of view of metric-measure geometry. For the sake of this introduction, a metric-measure space is a triple $(X,\sfd,\mm)$ where $(X,\sfd)$ is a compact metric space and $\mm$ is a Borel probability measure, playing the role of reference volume measure. Using optimal-transport techniques,  Sturm \cite{sturm:I, sturm:II} and Lott--Villani \cite{lottvillani:metric} introduced the curvature-dimension condition $\CD(K,N)$; the rough geometric picture is that a  $\CD(K,N)$ space should be thought of as a possibly non-smooth metric measure space with Ricci curvature bounded below by $K\in \R$ and dimension bounded above by $N\in (1,\infty)$ is a synthetic sense.  The basic idea of such a synthetic point of view is to analyse weighted convexity properties of certain entropy functionals along geodesics in the space of probability measures endowed with the quadratic transportation distance.
\\A key technical assumption throughout the paper  is the so called \emph{essentially non-branching} property \cite{RS2014}, 
which roughly amounts to require that the $L^{2}$-optimal transport between two absolutely continuous (with respect to the reference measure $\mm$) probability measures moves along a family of geodesics with no intersections, i.e.
a non-branching set of geodesics (for the precise definitions see Section \ref{s:W2}). Examples of essentially non-branching spaces are Riemannian manifolds, Alexandrov spaces, Ricci limits and more generally $\RCD(K,N)$-spaces, Finsler manifolds endowed with a strongly convex norm; a standard example of a space failing to satisfy the essential non-branching property is $\R^{2}$ endowed with the $L^{\infty}$ norm.
\\In the end of the introduction, when discussing the main steps of the proof, we will mention where the essentially non-branching property is used.
\\

Our approach to establish \eqref{main theorem intro} is to regard an $N$-dimensional Riemannian manifold $(M,g)$ with ${\rm Ric}_g\ge (N-1)g$ as an essentially non-branching  metric measure space $(X,\sfd,\mm)$ satisfying the $\CD(N-1,N)$ curvature-dimension condition, e.n.b. $\CD(N-1,N)$-space for short.  Considering this extension is natural, because the class of e.n.b.  $\CD(N-1,N)$-spaces contains measured Gromov-Hausdorff limits of $N$-dimensional Riemannian manifolds with Ricci tensor bounded from below by the constant $N-1$. And, in turn, a sequence of such Riemannian manifolds $(M_h,g_h)$ such that the right-hand side of \eqref{main theorem intro} tends to zero as $h\to\infty$ may develop singularities and admits a limit only in the measured Gromov-Hausdorff sense to an e.n.b. $\CD(N-1,N)$-space.  
\\

 In the enlarged class of  e.n.b. $\CD(N-1,N)$-spaces, round spheres are not anymore the only equality cases in the Levy-Gromov comparison Theorem, which instead coincide with the whole family of the so-called {\it spherical suspensions}. 
\\In addition, as proved in \cite{CM1} and recalled in Theorem \ref{theorem:LGM} below, the Levy-Gromov comparison theorem holds on essentially non-branching metric measure spaces
verifying the $\CD(N-1,N)$ condition with {\it any real number} $N> 1$. In this general setting, the comparison isoperimetric profile is the one defined by the model space
\begin{equation}\label{eq:Model1D}
\left([0,\pi], |\cdot|, \frac{\sin^{N-1}(t)}{\omega_{N}} \L^{1}\right)\qquad\mbox{where}\,\,\omega_N=\int_0^\pi\sin^{N-1}(t)\,dt\,,
\end{equation}
and $|\cdot|$ denotes the Euclidean distance on $\mathbb{R}$. Denoting by $\cI_{N-1,N,\pi}$ the isoperimetric profile of this comparison model space, see \eqref{isoperimetric profile mms}, we notice that
\[
\cI_{N-1,N,\pi}(v)=\cI_{(\mathbb{S}^N,g_{\mathbb{S}^N})}(v)\qquad\forall v\in(0,1)\,,\forall N\in\N\,,N\ge2\,.
\]
With this notation in force, we state our main theorem.

\begin{theorem} \label{thm:CD} For every real number $N>1$ and $v\in(0,1)$ there exists a real constant $C(N,v)>0$ with the following property. If $(X,\sfd,\mm)$ is an essentially non-branching metric measure space satisfying the $\CD(N-1,N)$ condition and $\mm(X) = 1$ with $\supp(\mm) = X$, then
\begin{equation}
  \label{tesi0}
  \pi-\diam(X)\le C(N,v)\,\Big(\cI_{(X,\sfd,\mm)}(v)-\cI_{N-1,N,\pi}(v)\Big)^{1/N}\,.
\end{equation}
Moreover, for every Borel set $E\subset X$ with $\mm(E)=v$ there exists $\bar{x}\in X$ such that
\begin{equation}
  \label{tesi}
  \mm(E\Delta B_{r_{N}(v)}(\bar{x})) \leq C(N,v) \;\Big(\PP(E)-\cI_{N-1,N,\pi}(\mm(E))\Big)^{\eta}\qquad \eta=\frac{N}{N^{2}+2N-1}\,,
\end{equation}
where $r_{N}(v)$ is defined by
\[
\int_{0}^{r_{N}(v)} \sin^{N-1}(t)dt = v \omega_{N}\,.
\]
Finally, if $(X,\sfd)$ is isometric to a smooth Riemannian manifold $(M,g)$ (endowed with any measure $\mm$ such that the assumptions of the theorem hold), one can take $\eta=N/(N^2+N-1)$ in \eqref{tesi}.
\end{theorem}
Let us first discuss the first claim \eqref{tesi0} by recalling the  celebrated {\it Myers Theorem} \cite{Myers}:  if $(M,g)$ is an $N$-dimensional smooth Riemannian manifold with ${\rm Ric}_g\ge (N-1)g$, then $\diam(M)\le\pi$. A refined version of Myers Theorem involving the Levy-Gromov isoperimetric deficit very similar to  \eqref{tesi0} was established by Berard-Besson-Gallot \cite{BBG},  still in the framework of smooth Riemannian manifold with ${\rm Ric}_g\ge (N-1)g$.
\\Theorem  \ref{thm:CD} shall be seen as a further step in two directions: first of all the space is not assumed to be smooth, secondly not only an estimate on the diameter of the space is expressed in terms of the Levy-Gromov isoperimetric deficit but also a quantitative estimate on the closeness of the competitor subset to the metric ball is established in  \eqref{tesi}. 
Indeed, as the reader will realize, the claim \eqref{tesi0} will be set along the way of proving the much harder  \eqref{tesi}. 

Inequality \eqref{tesi} naturally fits in the context of quantitative isoperimetric inequalities. The basic result in this area is the improved Euclidean isoperimetric Theorem proved in \cite{FuMP}, and stating that if $E\subset\R^n$ is a Borel set of positive and finite volume, then there exists $\bar{x}\in\R^n$ such that
\begin{equation}
  \label{euclidean case}
  \frac{|E\Delta B_{r_E}(\bar{x})|}{|E|}\le C(N)\,\Big(\frac{\PP(E)}{\PP(B_{r_E})}-1\Big)^{1/2}
\end{equation}
where $r_E$ is such that $|B_{r_E}|=|E|$; see also \cite{FiMP,CL}. A closer estimate to \eqref{tesi} is the improved spherical isoperimetric Theorem from \cite{bogelainduzaarfusco2}: this result actually {\it is} \eqref{tesi} in the special case that $(X,\sfd,\mm)=(\mathbb{S}^N,g_{\mathbb{S}^N})$
but with the sharp exponent $\eta=1/2$. 

Taking variations in the broad context of metric measure spaces makes the prediction on the sharp exponent $\eta$ of \eqref{tesi} an hard task. Even formulating a conjecture is challenging question and at the present stage 
it could actually be that $\eta={\rm O}(1/N)$ as $N\to\infty$ is already sharp. In the direction of this guess, we notice that the exponent $1/N$ in \eqref{tesi0} is indeed optimal in the class of metric measure spaces, as a direct computation on the model 1-dimensional space \eqref{eq:Model1D} shows.
\\Another question we do not address here is the explicit dependence of the constant $C(N,v)$ in Theorem \ref{thm:CD} with respect to the parameters $N\in (1,\infty)$ and $v\in (0,1)$; we just observe that, for a fixed $N\in (1,\infty)$ and a fixed $v_{0}\in (0,1/2]$ one has $\sup_{v\in [v_{0}, 1-v_{0}]} C(N,v)<\infty$. 
\\

A challenging feature of Theorem \ref{thm:CD} is that none of the three general methods to approach quantitative isoperimetry seems applicable in this context. This is evident for the approach in \cite{FuMP}, based on symmetrization inequalities. The approach developed in \cite{CL} to address \eqref{euclidean case}, and used in \cite{bogelainduzaarfusco2} to prove \eqref{tesi} with $\eta=1/2$ in the case $(X,\sfd,\mm)=(\mathbb{S}^N,g_{\mathbb{S}^N})$, has a vast domain of applicability. Essentially, the approach of \cite{CL} has a reasonable chance to work on every variational problem with a sufficiently smooth regularity theory and with strictly stable minimizers. (Depending on the problem, it may be quite non-trivial to implement one of, or both, these two points.) In our context, of course, there are no regularity theories and no second variation formulae to be exploited. Finally, the approach to \eqref{euclidean case}, and more generally to the quantitative Wulff inequality, developed in \cite{FiMP} is based on the Gromov-Knothe proof of the (Wulff) isoperimetric inequality \cite{Knothe,MS}. But, at present day, proving isoperimetry with the Gromov-Knothe argument beyond the case of Euclidean spaces seems to be an open problem: for example, to the best of our knowledge, it is not know how to adapt the Gromov-Knothe argument for proving the isoperimetric theorem on, say, the sphere.
\\

Before discussing the main steps of the proof  of Theorem \ref{thm:CD}, it  is worth including notable examples of spaces fitting in the assumptions of the result. Let us stress that our main theorem seems new  in all of them.
The class of essentially non branching  $\CD(N-1,N)$ spaces includes many remarkable family of spaces, among them:
\begin{itemize}
\item \emph{Measured Gromov Hausdorff limits of Riemannian $N$-dimensional manifolds  satisfying  ${\rm Ric}_g\ge (N-1)g$ and more generally the class of $\RCD(N-1,N)$ spaces}.
Indeed measured Gromov Hausdorff limits of Riemannian $N$-manifolds satisfying  ${\rm Ric}_g\ge (N-1)g$  are examples of $\RCD(N-1,N)$ spaces (see for instance \cite{GMS2013}) and $\RCD(N-1,N)$ spaces are essentially non-branching $\CD(N-1,N)$ (see \cite{RS2014}).
\item \emph{Alexandrov spaces with curvature $\geq 1$}.
Petrunin \cite{PLSV} proved that the lower curvature bound in the sense of comparison triangles is compatible with the optimal transport type lower bound on the  Ricci curvature given by Lott-Sturm-Villani (see also \cite{zhangzhu}).  Moreover  geodesics in Alexandrov spaces with curvature bounded below do not branch. It follows that Alexandrov spaces with curvature bounded from below by $1$ are non-branching  $\CD(N-1,N)$ spaces.
\item \emph{Finsler manifolds where the norm on the tangent spaces is strongly convex, and which satisfy lower Ricci curvature bounds.} More precisely we consider a $C^{\infty}$-manifold  $M$, endowed with a function $F:TM\to[0,\infty]$ such that $F|_{TM\setminus \{0\}}$ is $C^{\infty}$ and  for each $p \in M$ it holds that $F_p:=T_pM\to [0,\infty]$ is a  strongly-convex norm, i.e.
$$\qquad \quad \; g^p_{ij}(v):=\frac{\partial^2 (F_p^2)}{\partial v^i \partial v^j}(v) \quad \text{is a positive definite matrix at every } v \in T_pM\setminus\{0\}. $$
Under these conditions, it is known that one can write the  geodesic equations and geodesics do not branch: in other words these spaces are non-branching.
We also assume $(M,F)$ to be geodesically complete and endowed with a $C^{\infty}$ probability measure $\mm$ in a such a way that the associated m.m.s. $(X,F,\mm)$ satisfies the $\CD(N-1,N)$ condition. This class of spaces has been investigated by Ohta \cite{Ohta} who established the equivalence between the Curvature Dimension condition and a Finsler-version of Bakry-Emery $N$-Ricci tensor bounded from below.
\end{itemize}

We  conclude the introduction by  briefly illustrating  the main steps in the  proof of Theorem \ref{thm:CD}. The starting point of our approach is the metric measured version of the classical {\it localization technique}. First introduced in the study of sharp Poincar\'e inequalities on convex domains by Payne and Weinberger \cite{PW}, the localization technique has been developed into a general dimension reduction tool for geometric inequalities in symmetric spaces in the works of Gromov-Milman \cite{GrMi}, Lov\'asz-Simonovits \cite{LoSi} and Kannan-Lov\'asz-Simonovits \cite{KaLoSi}. More recently, Klartag  \cite{klartag} bridged the localization technique with Monge-Kantorovich optimal transportation problem, extending the range of applicability of the method to general Riemannian manifolds. The extension to the metric setting was finally obtained in \cite{CM1}, see Section \ref{Ss:localization}.

Given $E\subset X$, the localization Theorem (Theorem \ref{T:localize}) gives a decomposition of $X$ into a family of one-dimensional sets  $\{X_{q}\}_{q \in Q}$ formed by the transport rays of a Kantorovich potential associated to the optimal transport of (the normalized restriction of $\mm$ to) $E$ into its complement in $X$; each $X_{q}$ 
is in particular isometric to a real interval. 
A first crucial property of such a decomposition is that each ray $X_{q}$ carries a  natural measure $\mm_{q}$ (given by the the Disintegration Theorem) in such a way that $(X_{q}, \sfd, \mm_{q})$ is a $\CD(N-1,N)$ space and $\mm_{q}(E\cap X_{q})=\mm(E)$ so that both the geometry of the space and the constraint of the problem are \emph{localized} into a family of one-dimensional spaces. 
A key ingredient used in the proof of such a decomposition is the essentially non-branching property  which, coupled with $\CD(K,N)$ (actually the weaker measure contraction would suffice here), guarantees that the rays form a partition of $X$ (up to an $\mm$-negligible set).   

As a first step, we observe that most of such rays are sufficiently long (Proposition \ref{P:Q2-2}). This shows the first part of Theorem \ref{thm:CD}, that is, estimate \eqref{tesi0} (see Theorem \ref{T:Q2-3}).

A second crucial property of the decomposition $\{X_{q}\}_{q \in Q}$, inherited by the variational nature of the construction, is the so-called cyclical monotonicity. This is key to show that 
most of the transport rays $X_{q}$ have their starting point close to a ``south pole'' $\bar{x}$, and end-up nearby a ``north pole'' $\bar{y}$ (in particular, the distance between $\bar{x}$ and $\bar{y}$ is close to $\pi$) (Corollary \ref{C:positionSN}).   Then  we  observe that a one-dimensional version of Theorem \ref{thm:CD} (see Section \ref{S:QII}) forces
most of the fibers  $E_{q}:=E \cap X_{q}$ (that is the intersection of $E$ with the corresponding one dimensional  element of the partition) to be  $\L^{1}$ close to intervals centered either
at the ``north pole'' or at the ``south pole'' of $X_{q}$ (Lemma \ref{L:rayNS}).
To conclude the argument,  a delicate step is to show that either most of the fibers $E_{q}$  are starting from  the south pole or most of them are starting from the north pole.
In the smooth setting the proof can be obtained using a relative isoperimetric inequality.
In our general framework we have to give a self-contained argument (to overcome the lack of convex neighborhoods) using an additional localization.
\\

\noindent
We conclude with a few additional remarks.

First, although Theorem \ref{thm:CD} is formulated for $\CD(N-1,N)$ spaces, a statement for $\CD(K,N)$ spaces with $K>0$ is easily obtained by scaling. Indeed, $(X,\sfd,\mm)$ satisfies $\CD(K,N)$ if and only if, for any $\alpha, \beta\in (0,\infty)$, the scaled metric measure space $(X,\alpha \sfd, \beta \mm)$ satisfies $\CD(\alpha^{-2} K, N)$; see \cite[Proposition 1.4]{sturm:II}.

Second, it would be interesting to understand quantitative isoperimetry in metric measure spaces in the regime $N\to\infty$. The question is motivated by the validity of dimension independent quantitative isoperimetric estimates on Gaussian spaces (see \cite{CFMPGauss,MN,BJ1,BJ2} for a full account on this problem), and, of course, it is beyond the reach of Theorem \ref{thm:CD} as the exponent $\eta$ in \eqref{tesi} vanishes as $N\to\infty$.

Third, we recall that in \cite[Corollary 1.6]{CM1} the first and third author have proved the convergence to a spherical suspension (in the metric measured Gromov-Hausdorff sense) of any sequence of spaces $(X_i,\sfd_i,\mm_i)$ satisfying the $\CD(N-1,N)$ condition and such that $\cI_{(X_i,\sfd_i,\mm_i)}(v)\to\cI_{N-1,N,\pi}(v)$ for a fixed $v\in(0,1)$. It seem not obvious, from this information alone, to deduce the convergence of isoperimetric regions $E_i$ with $\mm(E_i)=v\,\mm(X_i)$ to geodesic balls in $X_i$ with radius $r_N(v)$. Thus, \eqref{tesi} in Theorem \ref{thm:CD}, besides being a quantitative estimate, provides a new information even without taking rates of convergence into account.

Fourth, in the smooth Riemannian case it is tempting to guess that using regularity theory one can bootstrap the $L^{1}$-estimate \eqref{tesi} into a smooth $C^{k}$-estimate. We wish to stress that this seems not so trivial: indeed for such an argument one would need a fixed Riemannian metric (which in this context would be the round metric on the sphere) or at least uniform $C^{k}$-estimates on the Riemannian metrics; this seems too much to hope for, as already proving a quantitative measured-Gromov Hausdorff estimate (which in turn is weaker than $C^{0}$-closeness of the metrics) seems challenging. 

Finally, we notice that the wide range of functional inequalities that can be  proved via the localization technique (see, e.g., \cite{CM2}) suggests a broad range of applicability for the constructions described in this paper.

\bigskip

\noindent{\bf Acknowledgement:} This work was supported by the NSF Grants DMS-1565354 and DMS-1361122. 
\\A.M. is  supported by the EPSRC First Grant EP/R004730/1 ``Optimal transport and geometric analysis''.

\section{Background material}\label{S:back} In this section we recall the main constructions needed in the paper. The reader familiar with curvature-dimension conditions and metric-measure spaces will just need to check Sections \ref{Ss:localization} and \ref{Ss:L1OT} for the decomposition of $X$ into transport rays (localization) which is going to be used throughout the paper. In Section \ref{s:W2} we review geodesics in the Wasserstein distance, in Section \ref{Ss:geom} curvature-dimension conditions, and in Section \ref{Ss:isoperimetric}   isoperimetric inequalities in the metric setting.

\subsection{Geodesics in the $L^2$-Wasserstein distance}\label{s:W2} A triple $(X,\sfd, \mm)$ is a metric measure space, m.m.s. for short, if $(X,\sfd)$ is a complete and separable metric space and $\mm$ a Borel non negative measure over $X$. We shall always assume that $\mm(X) =1$. The space of all Borel probability measures over $X$ will be denoted by $\mathcal{P}(X)$, while $\mathcal{P}_{2}(X)$ stands for the space of probability measures with finite second moment. On the space $\mathcal{P}_{2}(X)$ we define the $L^{2}$-Wasserstein distance $W_{2}$, by setting, for $\mu_0,\mu_1 \in \mathcal{P}_{2}(X)$,
\begin{equation}\label{eq:Wdef}
  W_2(\mu_0,\mu_1)^2 = \inf_{ \pi} \int_{X\times X} \sfd^2(x,y) \, \pi(dxdy)\,.
\end{equation}
Here the infimum is taken over all $\pi \in \mathcal{P}(X \times X)$ with $\mu_0$ and $\mu_1$ as the first and the second marginal, i.e. $(P_{1})_{\sharp} \pi= \mu_{0},  (P_{2})_{\sharp} \pi= \mu_{1}$. Of course $P_{i}, i=1,2$ is the projection on the first (resp. second) factor and $(P_{i})_{\sharp}$ denotes the corresponding push-forward map on measures.  As $(X,\sfd)$ is complete, also $(\mathcal{P}_{2}(X), W_{2})$ is complete.

Denote the space of geodesics of $(X,\sfd)$ by
$$
\Geo(X) : = \big\{ \gamma \in C([0,1], X):  \sfd(\gamma_{s},\gamma_{t}) = |s-t| \sfd(\gamma_{0},\gamma_{1}), \text{ for every } s,t \in [0,1] \big\}.
$$
Recall that a metric space is a geodesic space if and only if for each $x,y \in X$ there exists $\gamma \in \Geo(X)$ so that $\gamma_{0} =x, \gamma_{1} = y$. A basic fact on the $L^{2}$-Wasserstein distance, is that if $(X,\sfd)$ is geodesic, then $(\mathcal{P}_2(X), W_2)$ is geodesic.
Any geodesic $(\mu_t)_{t \in [0,1]}$ in $(\mathcal{P}_2(X), W_2)$  can be lifted to a measure $\nu \in {\mathcal {P}}(\Geo(X))$,
so that $({\rm e}_t)_\sharp \, \nu = \mu_t$ for all $t \in [0,1]$.
Here for any $t\in [0,1]$,  ${\rm e}_{t}$ denotes the evaluation map:
$$
  {\rm e}_{t} : \Geo(X) \to X, \qquad {\rm e}_{t}(\gamma) : = \gamma_{t}.
$$
Given $\mu_{0},\mu_{1} \in \mathcal{P}_{2}(X)$, we denote by
$\Opt(\mu_{0},\mu_{1})$ the space of all $\nu \in \mathcal{P}(\Geo(X))$ for which $({\rm e}_0,{\rm e}_1)_\sharp\, \nu$
realizes the minimum in \eqref{eq:Wdef}. If $(X,\sfd)$ is geodesic, then the set  $\Opt(\mu_{0},\mu_{1})$ is non-empty for any $\mu_0,\mu_1\in \mathcal{P}_2(X)$.

A set $F \subset \Geo(X)$ is a set of non-branching geodesics if and only if for any $\gamma^{1},\gamma^{2} \in F$, it holds:
$$
\exists \;  \bar t\in (0,1) \text{ such that } \ \forall t \in [0, \bar t\,] \quad  \gamma_{ t}^{1} = \gamma_{t}^{2}
\quad
\Longrightarrow
\quad
\gamma^{1}_{s} = \gamma^{2}_{s}, \quad \forall s \in [0,1].
$$
(Recall that a measure $\nu$ on a measurable space $(\Omega,\mathcal{F})$ is said to be concentrated
on $A \subset \Omega$ if $\exists B \subset A$ with $B \in \mathcal{F}$ so that $\nu(\Omega \setminus B) = 0$.) With this terminology, we recall from \cite{RS2014} the following definition.

\begin{definition}\label{D:essnonbranch}
A metric measure space $(X,\sfd, \mm)$ is \emph{essentially non-branching} if and only if for any $\mu_{0},\mu_{1} \in \mathcal{P}_{2}(X)$,
with $\mu_{0},\mu_{1}$ absolutely continuous with respect to $\mm$, any element of $\Opt(\mu_{0},\mu_{1})$ is concentrated on a set of non-branching geodesics.
\end{definition}

\subsection{Curvature-dimension conditions for metric measure spaces}\label{Ss:geom} The $L^2$-transport structure just described allows to formulate a generalized notion of Ricci curvature lower bound coupled with a dimension upper bound in the context of metric measure spaces. This is the $\CD(K,N)$ condition introduced in the seminal works of Sturm \cite{sturm:I, sturm:II} and Lott--Villani \cite{lottvillani:metric}, which here is reviewed only for a m.m.s. $(X,\sfd,\mm)$ with $\mm \in \mathcal{P}(X)$ and for $K > 0$ and $1 <N<\infty$ (the basic setting of the present paper).

For $N \in (1,\infty)$, the {\it $N$-R\'enyi relative-entropy functional}
$\mathcal{E}_N : \mathcal{P}(X) \rightarrow [0,1]$ is defined as
\[
\mathcal{E}_N(\mu) := \int \rho^{1 - \frac{1}{N}} d\mm \,,
\]
where $\mu = \rho \mm + \mu^{sing}$ is the Lebesgue decomposition of $\mu$ with $\mu^{sing} \perp \mm$.

\begin{definition}[$\tau_{K,N}$-coefficients] \label{def:sigma}
Given $K\in(0,\infty)$, $N \in(1,\infty)$, and $t \in [0,1]$, define $\sigma^{(t)}_{K,N}:[0,\infty)\to[0,\infty]$ by setting $\sigma^{(t)}_{K,N}(0) = t$,
$$
\sigma^{(t)}_{K,N}(\theta) := \frac{\sin(t \theta \sqrt{\frac{K}{N}})}{\sin(\theta \sqrt{\frac{K}{N}})}\qquad 0 < \theta < \frac{\pi}{\sqrt{K/N}}\,.
$$
and $\sigma^{(t)}_{K,N}(\theta) = +\infty$ otherwise; and define
\[
\tau_{K,N}^{(t)}(\theta) := t^{\frac{1}{N}} \sigma_{K,N-1}^{(t)}(\theta)^{1 - \frac{1}{N}}\,.
\]
\end{definition}

\begin{definition}[$\CD(K,N)$] \label{def:CDKN}
A m.m.s. $(X,\sfd,\mm)$ is said to satisfy $\CD(K,N)$ if for all $\mu_0,\mu_1 \in \mathcal{P}_2(X,\sfd,\mm)$,
there exists $\nu \in \Opt(\mu_0,\mu_1)$ so that for all $t\in[0,1]$, $\mu_t := ({\rm e}_{t})_{\#} \nu \ll \mm$, and for all $N' \geq N$:
\begin{equation} \label{eq:CDKN-def}
\mathcal{E}_{N'}(\mu_t) \geq \int_{X \times X} \left( \tau^{(1-t)}_{K,N'}(\sfd(x_0,x_1)) \rho_0^{-1/N'}(x_0) + \tau^{(t)}_{K,N'}(\sfd(x_0,x_1)) \rho_1^{-1/N'}(x_1) \right) \pi(dx_0,dx_1) ,
\end{equation}
where $\pi = ({\rm e}_0,{\rm e}_1)_{\sharp}(\nu)$ and $\mu_i = \rho_i \mm$, $i=0,1$.
\end{definition}

If $(X,\sfd,\mm)$ verifies the $\CD(K,N)$ condition then the same is valid for $(\supp[\mm],\sfd,\mm)$; hence we directly assume $X = \supp[\mm]$.

The following pointwise density inequality is a known equivalent definition of $\CD(K,N)$ on essentially non-branching spaces (the equivalence follows from  \cite{CM3}, see also \cite[Proposition 4.2]{sturm:II}).

\begin{definition}[$\CD(K,N)$ for essentially non-branching spaces] \label{def:CDKN-ENB}
An essentially non-branching m.m.s. $(X,\sfd,\mm)$ satisfies $\CD(K,N)$ if and only if for all $\mu_0,\mu_1 \in \mathcal{P}_2(X,\sfd,\mm)$, there exists a unique $\nu \in \Opt(\mu_0,\mu_1)$, $\nu$ is induced by a map (i.e. $\nu = S_{\sharp}(\mu_0)$ for some map $S : X \rightarrow \Geo(X)$),  $\mu_t := ({\rm e}_t)_{\#} \nu \ll \mm$ for all $t \in [0,1]$, and writing $\mu_t = \rho_t \mm$, we have for all $t \in [0,1]$:
\[
\rho_t^{-1/N}(\gamma_t) \geq  \tau_{K,N}^{(1-t)}(\sfd(\gamma_0,\gamma_1)) \rho_0^{-1/N}(\gamma_0) + \tau_{K,N}^{(t)}(\sfd(\gamma_0,\gamma_1)) \rho_1^{-1/N}(\gamma_1) \;\;\; \text{for $\nu$-a.e. $\gamma \in \Geo(X)$} .
\]
\end{definition}

For the general definition of $\CD(K,N)$ see \cite{lottvillani:metric, sturm:I, sturm:II}. It is worth recalling that if $(M,g)$ is a Riemannian manifold of dimension $n$ and
$h \in C^{2}(M)$ with $h > 0$, then the m.m.s.  $(M,\sfd_{g},h \, vol)$  verifies $\CD(K,N)$ with $N\geq n$ if and only if  (see Theorem 1.7 of \cite{sturm:II})
$$
Ric_{g,h,N} \geq  K g, \qquad Ric_{g,h,N} : =  Ric_{g} - (N-n) \frac{\nabla_{g}^{2} h^{\frac{1}{N-n}}}{h^{\frac{1}{N-n}}},
$$
 in other words if and only if the weighted Riemannian manifold $(M,g, h \, vol)$ has Bakry-\'Emery Ricci tensor bounded below by $K$.
Note that if $N = n$ the  Bakry-\'Emery  Ricci tensor $Ric_{g,h,N}= Ric_{g}$ makes sense only if $h$ is constant.
\medskip

We will use several times also the following terminology:
a non-negative function $h$ defined on an interval $I \subset \R$ is called a $\CD(K,N)$ density on $I$, for $K \in \R$ and $N \in (1,\infty)$, if for all $x_0,x_1 \in I$ and $t \in [0,1]$:
\begin{equation}\label{E:1dCD}
 h(t x_1 + (1-t) x_0)^{\frac{1}{N-1}} \geq  \sigma^{(t)}_{K,N-1}(| x_1-x_0|) h(x_1)^{\frac{1}{N-1}} + \sigma^{(1-t)}_{K,N-1}(|x_1-x_0|) h(x_0)^{\frac{1}{N-1}} ,
\end{equation}
(recalling the coefficients $\sigma$ from Definition \ref{def:sigma}).

The link with the definition of $\CD(K,N)$ for m.m.s. can be summarized as follows (\cite[Theorem A.2]{CMi}):
if $h$ is a $\CD(K,N)$ density on an interval $I \subset \R$ then the m.m.s. $(I,|\cdot |,h(t) dt)$ verifies $\CD(K,N)$; conversely, if the m.m.s. $(\R,|\cdot |,\mu)$
verifies $\CD(K,N)$ and $I = \supp(\mu)$ is not a point, then $\mu \ll \L^1$ and there exists a representant of the density $h = d\mu / d\L^1$ which is a $\CD(K,N)$ density on $I$.

In particular, if $I \subset \R$ is any interval, $h \in C^{2}(I)$, the m.m.s. $(I ,|\cdot|, h(t) dt)$ verifies $\CD(K,N)$ if and only if
\begin{equation}\label{E:CD-N-1}
\left(h^{\frac{1}{N-1}}\right)'' + \frac{K}{N-1}h^{\frac{1}{N-1}} \leq 0;
\end{equation}
see also Appendix \ref{Appendix A} for furhter properties of $\CD(K,N)$ densities.

\medskip

The lack of the local-to-global property of the $\CD(K,N)$ condition (for $K/N \neq 0$)
led in 2010 Bacher and Sturm to introduce in \cite{BS10} the reduced curvature-dimension condition, denoted by $\CD^{*}(K,N)$.
The $\CD^{*}(K,N)$ condition asks for the same inequality \eqref{eq:CDKN-def} of $\CD(K,N)$ to hold but
the coefficients $\tau_{K,N}^{(s)}(\sfd(\gamma_{0},\gamma_{1}))$ are replaced by the slightly smaller $\sigma_{K,N}^{(s)}(\sfd(\gamma_{0},\gamma_{1}))$.

A subsequent breakthrough in the theory was obtained with the introduction of the Riemannian curvature dimension condition $\RCD^{*}(K,N)$:
in the infinite dimensional case $N = \infty$  was introduced in \cite{AGS11b} for finite measures $\mm$ and in \cite{AGMR12} for $\sigma$-finite ones.
The class $\RCD^{*}(K,N)$ with $N<\infty$ (technically more involved) has been proposed in \cite{gigli:laplacian} and extensively investigated
in \cite{EKS,AMS}. We refer to these papers and references therein for a general account
on the synthetic formulation of the latter Riemannian-type Ricci curvature lower bounds.
Here we only briefly recall that it is a stable strengthening of the reduced curvature-dimension condition:
a m.m.s. verifies $\RCD^{*}(K,N)$ if and only if it satisfies $\CD^{*}(K,N)$ and is infinitesimally Hilbertian \cite[Definition 4.19 and Proposition 4.22]{gigli:laplacian}, 
meaning that the Sobolev space $W^{1,2}(X,\mm)$ is a Hilbert space (with the Hilbert structure induced by the Cheeger energy).

To conclude we recall also that recently, the first named author together with E. Milman  in \cite{CMi} proved the equivalence
of $\CD(K,N)$ and $\CD^{*}(K,N)$ (and also of the $\CD^{e}(K,N)$ and $\CD^{1}(K,N)$),
together with the local-to-global property for $\CD(K,N)$, in the framework of  essentially non-branching m.m.s. having $\mm(X) < \infty$.
As we will always assume the aforementioned properties to be satisfied by our ambient m.m.s. $(X,\sfd,\mm)$, we will use both formulations with no distinction.
It is worth also mentioning that a m.m.s. verifying $\RCD^{*}(K,N)$ is essentially non-branching (see \cite[Corollary 1.2]{RS2014})
implying also the equivalence of $\RCD^{*}(K,N)$ and  $\RCD(K,N)$ (see \cite{CMi} for details).

We shall always assume that the m.m.s. $(X,\sfd,\mm)$ is essentially non-branching and satisfies $\CD(K,N)$ from some $K>0$
with $\supp(\mm) = X$. It follows that $(X,\sfd)$ is a geodesic and compact metric space.

\subsection{Isoperimetric inequality for metric measure spaces}\label{Ss:isoperimetric}

In \cite{CM1} the L\'evy-Gromov-Milman isoperimetric inequality has been obtained for an essentially non-branching m.m.s. $(X,\sfd,\mm)$ verifying $\CD(K,N)$
with $\mm(X) = 1$; also the rigidity statement has been obtained in the smaller class of $\RCD(K,N)$ spaces.
What follows is a short overview of the statements as obtained in the
subsequent \cite{CM5} where the results of \cite{CM1} are obtained replacing the outer Minkowski content with the perimeter functional;
see also \cite{ADMG17}  for the general relation between the outer Minkowski content with the perimeter functional.

Denote by $\Lip(X)$ the space of real-valued Lipschitz functions over $X$. Given $u \in \Lip(X)$ its slope $|\nabla u|(x)$ at $x\in X$ is defined by
\begin{equation}
|\nabla u|(x):=\limsup_{y\to x} \frac{|u(x)-u(y)|}{\sfd(x,y)}.
\end{equation}
Following \cite{Am1,Am2,Mir} and the more recent \cite{ADM},
given a Borel subset $E \subset X$ and $A$ open, the perimeter of $E$ relative to $A$ is denoted by $\mathsf{P}(E,A)$ and is defined as follows
\begin{eqnarray}
\PP(E,A)&:=& \inf\left\{\liminf_{n\to \infty} \int_A |\nabla u_n| \,\mm \,:\,  u_n \in \Lip(A), \, u_n\to \chi_E \text{ in } L^1(A,\mm)\right\}.  \label{eq:defP}
\end{eqnarray}
We say that $E \subset X$ has finite perimeter in $X$ if $\PP(E,X) < \infty$.
We recall also few properties of the perimeter functions:
\begin{itemize}
\item[(a)] (locality) $\PP(E,A) = \PP(F,A)$, whenever $\mm((E\Delta F) \cap A) = 0$;
\item[(b)] (l.s.c.) the map $E \mapsto \PP(E,A)$ is lower-semicontinuous with respect to the $L^{1}_{loc}(A)$ convergence;
\item[(c)] (complementation) $\PP(E,A) = \PP(E^{c},A)$.
\end{itemize}
Moreover if $E$ is a set of finite perimeter, then the set function $A \to \PP(E,A)$ is the restriction to open sets of a finite
Borel measure $\PP(E,\cdot)$ in $X$ (see Lemma 5.2 of \cite{ADM}), defined by
$$
\PP(E,B) : = \inf \{ \PP(E,A) \colon A \supset B, \ A\ \textrm{open} \}.
$$
Sometimes, for ease of notation, we will write $\PP(E)$ instead of $\PP(E,X)$.

\medskip

The \emph{isoperimetric profile function} of $(X,\sfd,\mm)$, denoted by  ${\cI}_{(X,\sfd,\mm)}$,
is defined as the point-wise maximal function so that $\PP(A)\geq \cI_{(X,\sfd,\mm)}(\mm(A))$ for every Borel set $A \subset X$, that is
\begin{equation}
  \label{isoperimetric profile mms}
  \cI_{(X,\sfd,\mm)}(v) : = \inf \big\{ \PP(A) \colon A \subset X \, \textrm{ Borel}, \, \mm(A) = v   \big\}.
\end{equation}

\begin{theorem}[L\'evy-Gromov-Milman in $\CD(K,N)$-spaces, \cite{CM1, CM5}]\label{theorem:LGM}
Let $(X,\sfd,\mm)$ be an essentially non-branching metric measure space with $\mm(X)=1$ and having diameter  $D\in (0,+\infty]$.
Assume it satisfies the $\CD(K,N)$ condition  for some $K\in \R, N \in (1,\infty)$. Then  for every Borel set $E\subset X$ it holds
$$
\PP(E)\geq \cI_{K,N,D}(\mm(E)),
$$
where $\cI_{K,N,D}$ are the model isoperimetric profile functions obtained in  \cite{Mil}, i.e.
 $\cI_{(X,\sfd,\mm)}(v)\geq \cI_{K,N,D}(v)$ for every $v \in [0,1]$.
\medskip

If $(X,\sfd,\mm)$ satisfies $\RCD(N-1,N)$ for some  $N \in [2,\infty)$
and there exists $\bar{v} \in (0,1)$ such that $\cI_{(X,\sfd,\mm)}(\bar{v})=\cI_{N-1,N,\infty}(\bar{v})$,
then $(X,\sfd,\mm)$ is a spherical suspension:  there exists
an $\RCD(N-2,N-1)$ space $(Y,\sfd_{Y}, \mm_{Y})$ with $\mm_{Y}(Y)=1$ such that  $X$ is isomorphic as metric measure space to $[0,\pi] \times^{N-1}_{\sin} Y$.
\end{theorem}

As reported above, the model spaces for general $K,N$ have been discovered by E. Milman \cite{Mil} who extended the L\'evy-Gromov isoperimetric inequality to smooth manifolds with densities, i.e. smooth Riemannian manifold whose volume measure has been multiplied by a smooth non negative integrable density function.
Milman detected a model isoperimetric profile  $\cI_{K,N,D}$ such that if a Riemannian manifold with density has diameter at most $D>0$, generalized Ricci curvature at least $K\in \R$ and generalized dimension at most $N\geq 1$ then the isoperimetric profile function of the weighted manifold is bounded below by   $\cI_{K,N,D}$.

During the paper, we will make extensive use of of $\cI_{K,N,D}$, at least in the case $K>0$;
we now therefore review their definitions (and refer to \cite{Mil} for all the details, see in particular Theorem 1.2 and Corollary A.3):
\begin{itemize}

\item \textbf{Case 1}: $K>0$ and $D<\sqrt{\frac{N-1}{K}} \pi$,
$$
\cI_{K,N,D}(v) =   \min_{\xi \in \big[0, \sqrt{\frac{N-1}{K}} \pi -D\big]} \cI_{\big( [\xi,\xi +D],  \sin( \sqrt{\frac{K}{N-1}} t)^{N-1} \big)}(v), \quad \forall v \in [0,1] ~;
$$

\item \textbf{Case 2}:   $K > 0$ and $D \geq \sqrt{\frac{N-1}{K}} \pi$,
$$
\cI_{K,N,D}(v) = \cI_{ \big( [0, \sqrt{\frac{N-1}{K}} \pi],  \sin( \sqrt{\frac{K}{N-1}} t)^{N-1}   \big)}(v), \quad \forall v \in [0,1] ~;
$$
\end{itemize}
where in both cases we have used the following notation: given $f$ on a closed interval $L \subset \R$, we denote with $\mu_{f,L}$
the probability measure supported in $L$ with density (with respect to the Lebesgue measure) proportional to $f$ there and
$\cI_{(L,f)}$ stands for $\cI_{(L,\, |\cdot|, \mu_{f,L})}$.
Note that when $N$ is an integer,
$$
\cI_{\big( [0,  \sqrt{\frac{N-1}{K}} \pi ], ( \sin(\sqrt{\frac{K}{N-1} } t)^{N-1}\big)} = \cI_{({\mathbb S}^{N}, g^K_{can}, \mu^K_{can})},
$$
by the isoperimetric inequality on the sphere, and so Case 2 with $N$ integer corresponds to L\'evy-Gromov isoperimetric inequality.
  \\ In order to keep the notation short we will often write $\cI_{D}$ in place of $\cI_{N-1, N,D}$.


\subsection{Localization}\label{Ss:localization}
Theorem \ref{theorem:LGM} has been proved obtaining a dimensional reduction of the isoperimetric inequality via the so-called ``Localization theorem'',
proved for essentially non-branching metric measure spaces verifying the $\CD(K,N)$ condition.

The localization theorem has its roots in a work of  Payne-Weinberger \cite{PW} and has been developed by Gromov-Milman \cite{GrMi}, Lov\'asz-Simonovits \cite{LoSi} and Kannan-Lov\'asz-Simonovits \cite{KaLoSi}, and consists in reducing an $n$-dimensional problem, via tools of convex geometry,
to lower dimensional problems that one can handle.  In the previous papers the symmetric properties of the spaces were necessary to obtain such a dimensional reduction.
In the recent paper \cite{klartag}, Klartag found a bridge  between $L^1$-optimal transportation problems and the localization techinque yielding the localization theorem in the framework of smooth Riemannian manifolds.
Inspired by this approach, the first and the third author in \cite{CM1}
proved the following localization theorem for essentially non-branching metric measure spaces verifying the $\CD(K,N)$ condition.

\begin{theorem}[\cite{CM1}]\label{T:localize}
Let $(X,\sfd, \mm)$ be an essentially non-branching metric measure space verifying the $\CD(K,N)$ condition for some $K\in \R$ and $N\in [1,\infty)$.
Let $f : X \to \R$ be $\mm$-integrable such that $\int_{X} f\, \mm = 0$ and assume the existence of $x_{0} \in X$ such that $\int_{X} | f(x) |\,  \sfd(x,x_{0})\, \mm(dx)< \infty$.
\medskip

Then the space $X$ can be written as the disjoint union of two sets $Z$ and $\mathcal{T}$ with $\mathcal{T}$ admitting a partition
$\{ X_{q} \}_{q \in Q}$ and a corresponding disintegration of $\mm\llcorner_{\mathcal{T}}$, $\{\mm_{q} \}_{q \in Q}$ such that:

\begin{itemize}
\item For any $\mm$-measurable set $B \subset \mathcal{T}$ it holds
$$
\mm(B) = \int_{Q} \mm_{q}(B) \, \qq(dq),
$$
where $\qq$ is a probability measure over $Q$ defined on the quotient $\sigma$-algebra $\mathcal{Q}$.
\medskip
\item For $\qq$-almost every $q \in Q$, the set $X_{q}$ is a geodesic and $\mm_{q}$ is supported on it.
Moreover $q \mapsto \mm_{q}$ is a $\CD(K,N)$ disintegration.
\medskip
\item For $\qq$-almost every $q \in Q$, it holds $\int_{X_{q}} f \, \mm_{q} = 0$ and $f = 0$ $\mm$-a.e. in $Z$.
\end{itemize}
\end{theorem}

We refer to Appendix \ref{Appendix B} for the Disintegration Theorem and its link with partitions of the space.
Here we only mention that $q \mapsto \mm_{q}$ is a $\CD(K,N)$ disintegration has to be understood as follows:
for $\qq$-a.e. $q \in Q$, $\mm_{q} = h_{q} \mathcal{H}^{1}\llcorner_{X_{q}}$, where $\mathcal{H}^{1}$ denotes the one-dimensional Hausdorff measure
and $h_{q} \circ X_{q}$  is a $\CD(K,N)$ density, in the sense of \eqref{E:1dCD}; here,  with a slight abuse of notation, $X_{q}$ denotes also the map with image $X_{q}$.

In the next section we recall all the needed terminology and objects from the theory of $L^{1}$-optimal transportation
used to obtain Theorem \ref{T:localize}.
This will also serve as basis for establishing the main result of the present paper i.e. a quantitative isoperimetric inequality.


\subsection{$L^{1}$ optimal transportation}\label{Ss:L1OT}

In this section we recall only some facts from the theory of $L^1$ optimal transportation which are of some interest for this paper;
we refer to \cite{ambro:lecturenote, AmbrosioPratelliL1, biacava:streconv, cava:MongeRCD, CMi, EvansGangbo,FeldmanMcCann-Manifold, klartag, Vil:topics}
and references therein for more details on the theory of $L^1$ optimal transportation.

Following  the approach of \cite{klartag}, Theorem \ref{T:localize} has been proven in \cite{CM1} studying the following optimal transportation problem: define
$\mu_{0} : = f^{+} \mm$ and $\mu_{1} : = f^{-}\mm$, where $f^{\pm}$ denote the positive and the negative part of $f$, respectively, and study the
$L^{1}$-optimal transport problem associated with it
\begin{equation}\label{eq:defL1OT}
\inf \left\{ \int_{X\times X} \sfd(x,y) \, \pi(dxdy) \colon \pi \in \mathcal{P}(X\times X), \ (P_{1})_{\sharp}\pi = \mu_{0}, (P_{2})_{\sharp}\pi = \mu_{1} \right\};
\end{equation}
where $P_{i}$ denotes the projection onto the $i$-th component.
Then the relevant object to study is given by the dual formulation of the previous minimization problem.
By the summability properties of $f$ (see the hypothesis of Theorem \ref{T:localize}), there exists a $1$-Lipschitz function $\f : X \to \R$  such that
$\pi$ is a minimizer in \eqref{eq:defL1OT} if and only if $\pi(\Gamma) = 1$, where
$$
\Gamma : = \{ (x,y) \in X \times X  \colon \f(x) - \f(y) = \sfd(x,y)\}
$$
is the naturally associated $\sfd$-cyclically monotone set, i.e. for any $(x_{1},y_{1}), \dots, (x_{n},y_{n}) \in \Gamma$ it holds
$$
\sum_{i = 1}^{n} \sfd(x_{i},y_{i}) \leq \sum_{i = 1}^{n} \sfd(x_{i},y_{i+1}),  \qquad y_{n+1} = y_{1},
$$
for any $n \in \N$.
The set $\Gamma$ induces a partial order relation whose maximal chains produce a partition made of one dimensional sets of a certain subset of the space,
provided the ambient space $X$ verifies some mild regulartiy properties.

We now review how to obtain the partition from $\Gamma$; this procedure has been already presented and used in several contributions
(\cite{AmbrosioPratelliL1, biacava:streconv, FeldmanMcCann-Manifold, klartag, Vil:topics})
when the ambient space is the euclidean space, a manifold or a non-branching metric space (see \cite{biacava:streconv, cava:Wiener} for extended metric spaces);
the analysis in our framework started with \cite{cava:MongeRCD} and has been refined and extended in \cite{CMi};
we will follow the notation of \cite{CMi} to which we refer for more details.

The \emph{transport relation} $R$ and the \emph{transport set} $\mathcal{T}$ are defined  as:
\begin{equation}\label{E:R}
R := \Gamma \cup \Gamma^{-1} = \{ |\f(x) - \f(y)| = \sfd(x,y) \} ~,~ \mathcal{T} := P_{1}(R \setminus \{ x = y \}) ,
\end{equation}
where $\{ x = y\}$ denotes the diagonal $\{ (x,y) \in X^{2} : x=y \}$ and $\Gamma^{-1}= \{ (x,y) \in X \times X : (y,x) \in \Gamma\}$.
Since $\f$ is $1$-Lipschitz, $\Gamma, \Gamma^{-1}$ and $R$ are closed sets and therefore, from the compactness of $(X,\sfd)$ (recall $\CD(K,N)$ with $K>0$), compact;
consequently $\mathcal{T}$ is $\sigma$-compact.

It is immediate to verify (see \cite[Proposition 4.2]{ambro:lecturenote}) that
if $(\gamma_0,\gamma_1) \in \Gamma$ for some  $\gamma \in \Geo(X)$, then $(\gamma_{s},\gamma_{t}) \in \Gamma$  for all $0\leq s \leq t \leq 1$.
To exclude possible branching we need to consider the following sets, introduced in \cite{cava:MongeRCD}:
\begin{align*}
	A_{+}	: = 	&~\{ x \in \mathcal{T} : \exists z,w \in \Gamma(x), (z,w) \notin R \}, \nonumber \\
	A_{-}		: = 	&~\{ x \in \mathcal{T} : \exists z,w \in \Gamma^{-1}(x), (z,w) \notin R \};
\end{align*}
where $\Gamma(x) = \{y \in X \; ;\; (x,y) \in \Gamma\}$ denotes the section of $\Gamma$ through $x$ in the first coordinate, and similarly for $R(x)$ (through either coordinates by symmetry).
$A_{\pm}$ are called the sets of forward and backward branching points, respectively. Note that both $A_{\pm}$ are $\sigma$-compact sets.
Then the non-branched transport set has been defined as
$$
\T^{b} : = \T \setminus (A_{+} \cup A_{-}),
$$
and is a Borel set; accordingly the  non-branched transport relation is given by:
$$
R^b := R \cap (\T^b \times \T^b) .
$$
In was shown in \cite{cava:MongeRCD} (cf. \cite{biacava:streconv}) that $R^b$ is an equivalence relation
over $\T^{b}$ and that for any $x \in \T^{b}$, $R(x) \subset (X,\sfd)$ is isometric to a closed interval in $(\R, | \cdot |)$.

Now, from the first part of the Disintegration Theorem (see Theorem \ref{T:disintegrationgeneral}) applied to $(\T^b , \mathcal{B}(\T^b), \mm\llcorner_{\T^{b}})$, we obtain an essentially unique disintegration of $\mm\llcorner_{\T^{b}}$ consistent with the partition
of $\T^{b}$ given by the equivalence classes $\{ R^b(q)\}_{q \in Q}$ of $R^{b}$:
$$
\mm\llcorner_{\T^{b}} = \int _{Q} \mm_{q}\,\qq(dq ),
$$
with corresponding quotient space $(Q, \mathscr{Q},\qq)$ ($Q \subset \T^b$ may be chosen to be any section of the above partition).
In what follows, we will use also the notation $X_{q}$ to denote the transport ray $R^{b}(q)$.

The next step is to show that the disintegration is strongly consistent. By the  Disintegration Theorem, this is equivalent to the existence of a $\mm\llcorner_{\T^{b}}$-section $\bar Q \in \mathcal{B}(\T^b)$ (which by a mild abuse of notation we will call $\mm$-section),
such that the quotient map associated to the partition is $\mm$-measurable,
where we endow $\bar Q$ with the trace $\sigma$-algebra.
This has already been shown in \cite[Proposition 4.4]{biacava:streconv} in the framework of non-branching metric spaces;
since its proof does not use any non-branching assumption, we can conclude that:
$$
\mm\llcorner_{\T^{b}} = \int _{Q} \mm_{q}\,\qq(dq ), \quad \text{and for } \qq-\text{a.e. } q \in Q, \quad \mm_{q}(R^b(q)) =1,
$$
where now $Q \supset \bar Q \in \mathcal{B}(\T^b)$ with $\bar Q$ an $\mm$-section for the above partition (and hence $\qq$ is concentrated on $\bar Q$).
Moreover the existence of an $\mm$-measurable quotient map permits to conclude that the quotient $\sigma$-algebra on $\bar Q$, that we denote with
$\mathscr{Q} \cap \bar Q$, is contained in $\overline{\mathcal{B}(\bar Q)}^{\qq}$, the completion with respect to $\qq$ of the Borel $\sigma$-algebra over $\bar Q$.

The existence of an $\mm$-section also permits to construct a measurable parametrization of the transport rays.
We can define
$$
g : \dom(g) \subset \bar Q \times \R \to \T^{b}
$$
that associates to $(q,t)$ the unique $x \in \Gamma(q)$ with $\sfd(q,x) = t$, provided $t > 0$,
or the unique $x \in \Gamma^{-1}(q)$ with $\sfd(q,x) = -t $, otherwise.
Then
\begin{align*}
\gr(g) = &~ \{ (q,t,x) \in \bar Q \times [0,\infty) \times \T^{b} \colon (q,x) \in \Gamma, \ \sfd(q,x) = t \}  \\
&~ \cup \{ (q,t,x) \in \bar Q \times (-\infty, 0) \times \T^{b} \colon (q,x) \in \Gamma^{-1}, \ \sfd(q,x) = - t \} ,
\end{align*}
showing that $\gr(g)$ is Borel; in particular $g : \dom(g) \to \T^{b}$ is a Borel map with $\dom(g)$ analytic set and
 image  $\cup_{\alpha \in \bar Q} R^{b}(\alpha)$, that is analytic as well.
To conclude we also notice that $g$ is injective and
$$
\mm(\T^{b} \setminus \bigcup_{q \in \bar Q} R^{b}(q) ) =0.
$$

\medskip
A-priori the non-branched transport set $\T^b$ can be much smaller than $\T$. However, under fairly general assumptions one can prove that the sets
$A_{\pm}$ of forward and backward branching are both $\mm$-negligible.
In \cite{cava:MongeRCD} this was shown for a m.m.s. $(X,\sfd,\mm)$ verifying $\RCD(K,N)$ and $\supp(\mm) = X$.
The proof only relies on the following two properties which hold for the latter spaces:
\begin{itemize}
\item[-] $\supp(\mm) = X$.
\item[-] Given $\mu_{0}, \mu_{1} \in \mathcal{P}_{2}(X)$ with $\mu_{0}\ll \mm$,
there exists a unique optimal transference plan for the $W_{2}$-distance and it is induced by an optimal transport \emph{map}.
\end{itemize}
These properties are also verified for an essentially non-branching m.m.s. $(X,\sfd,\mm)$ satisfying $\CD(K,N)$ and $\supp(\mm)= X$ (see \cite{CM3}).

We summarize the above discussion in:

\begin{corollary} \label{C:disint}
Let $(X,\sfd,\mm)$ be an essentially non-branching m.m.s. satisfying $\CD(K,N)$ and $\supp(X) = \mm$.
Then for any $1$-Lipschitz function $\f : X \to \R$, we have $\mm(\T \setminus \T^b) = 0$. In particular, we obtain the following essentially unique disintegration $(Q,\mathscr{Q},\qq)$ of $\mm\llcorner_{\T} = \mm\llcorner_{\T^b}$ strongly consistent with the partition
of $\T^{b}$ given by the equivalence classes $\{R^b(q)\}_{q \in Q}$ of $R^{b}$:
\begin{equation}\label{E:disintMCP}
\mm\llcorner_{\T} = \int_{Q} \mm_{q} \,\qq(dq), \quad \text{and for } \qq-\text{a.e. } q \in Q, \quad \mm_{q}(R^b(q)) =1 .
\end{equation}
Here $Q$ may be chosen to be a section of the above partition so that $Q \supset \bar Q \in \mathcal{B}(\T^b)$ with $\bar Q$ an $\mm$-section with $\mm$-measurable quotient map. In particular, $\mathscr{Q} \supset \mathcal{B}(\bar Q)$ and $\qq$ is concentrated on $\bar Q$.
\end{corollary}

To obtain Theorem \ref{T:localize} one still needs to show that the constraint $\int_{X} f \, \mm= 0$ is localized,  i.e. $\int_{X_{q}} f \,\mm_{q} = 0$ for $\qq$-a.e. $q\in Q$,
together with the curvature bound: $q \mapsto \mm_{q}$ is a $\CD(K,N)$ disintegration. The first property it verified almost ``by contruction'';
the second one is the more subtle and to prove it one should study the interplay between $L^{2}$-Wasserstein geodesics and the transport set $\T$;
we refer to \cite[Theorem 4.2]{CM1} for all the details.

\medskip

Finally, we recall that $\qq$-a.e. $R^{b}(q)$ is actually maximal, meaning that it coincides with $R(q)$.
This can be restated as follows. $\Gamma$ induces a partial order relation on $X$:
$$
y \leq x \;\;\; \Leftrightarrow \;\;\; (x,y) \in \Gamma,
$$
and for $x \in \T^{b}$,  $(R(x),\sfd)$ is isometric to a closed interval in $(\R,|\cdot |)$.
This isometry induces a total ordering on $R(x)$ which must coincide with either $\leq$ or $\geq$, implying that $(R(x), \leq)$ is totally ordered;
in particular it is a chain; the previous maximality property means that $R^b(q) = R(q) \cap \T^b$ is a maximal chain in the partially ordered set  $(X,\leq)$.

To rigorously state this property, we use the classical definition of initial and final points, $\mathcal{A}$ and $\mathcal{B}$, respectively:
\begin{align*}
\mathcal{A} : = &~\{ x \in \T \colon  \nexists y \in \T, \ (y,x) \in \Gamma, \ y \neq x  \}, \\
\mathcal{B} : = &~\{ x \in \T \colon  \nexists y \in \T, \ (x,y) \in \Gamma, \ y \neq x  \}.
\end{align*}
Note that:
$$
\mathcal{A} = \T \setminus P_{1}\big(\{ \Gamma \setminus \{ x= y \} \}) ,
$$
so $\mathcal{A}$ is Borel (since $(X,\sfd)$ is compact the set $\Gamma \setminus \{ x= y \} $ is $\sigma$-compact); similarly for $\mathcal{B}$.

\begin{theorem}[\cite{CMi}]\label{T:endpoints}
Let $(X,\sfd, \mm)$ be an essentially non-branching m.m.s. verifying $\CD(K,N)$ and $\supp(\mm) = X$. Let $\f : (X,\sfd) \rightarrow \R$ be
any $1$-Lipschitz function, with \eqref{E:disintMCP} the associated disintegration of $\mm \llcorner_{\T}$.  \\
Then there exists
 $\hat Q \subset Q$ such that $\qq(Q \setminus \hat Q) = 0$ and for any $q \in \hat Q$ it holds:
$$
R(q) \setminus \T^{b} \subset \mathcal{A} \cup \mathcal{B}.
$$
In particular, for every $q \in \hat Q$:
\[
R(q) = \overline{R^b(q)} \supset R^b(q) \supset \mathring{R}(q) ,
\]
(with the latter interpreted as the relative interior).
\end{theorem}

Possibly taking a  full $\qq$-measure  subset of $\hat Q$, we can assume $\hat Q$ to be Borel. During the paper we will make use of the map associating to each point
$q \in \hat Q$ the starting point of the ray and the end point of the ray. As we are assuming $\CD(K,N)$ with $K>0$, we will think of the starting point as
a ``south pole'' and the ending point as a ``north pole''; this justifies the following notation
$$
P_{S} : \hat Q \to \T, \qquad P_{N} : \hat Q \to \T
$$
with graphs
$$
\gr(P_{S}) : = (\hat Q \times \mathcal{A}) \cap \Gamma^{-1}, \qquad \gr(P_{N}) : = (\hat Q \times \mathcal{B}) \cap \Gamma,
$$
also showing that both $P_{S}$ and $P_{N}$ are Borel maps;
this implies that also the map  $\hat Q \ni q \mapsto |X_{q}|  = \sfd(P_{S}(q),P_{N}(q))$,
is Borel.

From the measurability of the disintegration, one also obtains that
$$
\hat Q \times \R \ni (q,t) \mapsto h_{q} (t) \in [0,\infty),
$$
is $\qq\otimes\L^{1}$-measurable (see for instance \cite[Proposition 10.4]{CMi}); here with an abuse of notation we have denoted with $h_{q}$
the density function $h_{q} \circ g(q,\cdot)$ where $\mm_{q} = h_{q} \mathcal{H}^{1}\llcorner_{X_{q}}$.
It is fairly standard (using for example \cite[Theorem 3.1.30]{Srivastava}) to restrict ourselves to a Borel subset of $\hat Q$ of the same $\qq$-measure,
that for ease of notation we denote again with $\hat Q$, such that
$$
\hat Q \times \R \ni (q,t) \mapsto h_{q} (t) \in [0,\infty),
$$
is Borel. Then we can compose it with a translation, Borel in $q$, to obtain that for each $q \in \hat Q$, $h_{q} : [0,|X_{q}|] \to [0,\infty)$,
and still obtain a jointly Borel function.

One can also restrict the Borel map $g$ to the following Borel subset of its domain:
$$
\{ (q,s) \in \hat Q \times \R \colon s \in (0,|X_{q}|)\},
$$
and by construction, the restriction of $g$ is injective.

We conclude this part mentioning that we will directly write $Q$ instead of $\hat Q$
and we summarize the measurability properties obtained:

\begin{itemize}
\item[-] The disintegration formula holds: for  a suitably chosen  $Q \subset \T^{b}$ Borel,  it holds
$$
\mm\llcorner_{\T}  = \int_{Q} \mm_{q} \, \qq(dq), \quad \text{and} \quad  \mm_{q}(R(q)) = 1,\;  \qq\text{-a.e. } q \in Q;
$$
\item[-] For $\qq$-a.e. $q \in Q$, $\mm_{q} = (g(q,\cdot))_{\sharp} \, h_{q} \L^{1}\llcorner_{[0,|X_{q}|]}$ with $h_{q}$ a $\CD(K,N)$ density
and the maps $g : \dom(g) \to X$, $h : \dom (h) \to [0,\infty)$ with $\dom (g), \dom (h) \subset Q \times \R$ are Borel measurable.
\end{itemize}


\section{Quantitative one-dimensional estimates}\label{S:onedstab}

In this section we obtain all the one-dimensional results concerning the quantitative isoperimetric inequality that will then be used
in the general framework of metric measure spaces.

We start considering the one-dimensional metric measure space $([0,D], |\cdot|, h\cdot \L^{1})$ verifying $\CD(N-1,N)$, i.e.
$$
\big( h^{1/(N-1)} \big)'' + h^{1/(N-1)} \leq 0,
$$
in the sense of distributions and such that $\int_{[0,D]} h(t)dt = 1$; notice that,  as by construction $h \geq 0$, then  necessarily $h > 0$ over $(0,D)$.
 Since   $\CD(N-1,N)$ densities are log-concave (meaning that $-\log h$ is convex), by a result of Bobkov \cite[Proposition 2.1]{Bobkov} for each $v \in (0,1)$ there exists an  isoperimetric minimizer  either of the form $[0,r^{-}_{h}(v)]$ or $[r^{+}_{h}(v), D]$,
where
$$
v =  \int_{0}^{r^{-}_{h}(v)} h(t) \,dt =  \int_{r^{+}_{h}(v)}^{D} h(t) \,dt.
$$
The perimeter functional associated to $h$ will be denoted by $\PP_{h}$.
Then it is natural to define the deficit associated to $h$ as follows
\begin{equation}\label{eq:defdelta2}
\delta_{h}(E) : =  \PP_{h}(E) - \mathcal{I}_{h}(v), \qquad \mathcal{I}_{h}(v) : = \min\{ h(r^{-}_{h}(v)), h(r^{+}_{h}(v)) \} .
\end{equation}

Then, calling $\mm : = h \cdot \L^{1}$, we obtain the next quantitative statement.
\begin{proposition}\label{P:1dquant}
For each $v \in (0,1)$ there exists  $\ve(N,v), C(N,v)>0$  such that for each $\ve : = \pi - D \in (0, \ve(N,v))$
$$
\PP_{h}(E) - \I_{h}(v) \geq C(N,v) \min\{ \mm( E \Delta [0,r_{h}^{-}(v)]), \mm( E \Delta [r_{h}^{+}(v),D])    \},
$$
for any $\mm = h \L^{1}$ with $h$ a $\CD(N-1,N)$ density, supported over $[0,D]$ that integrates to 1 and $E \in \mathcal{B}(\R)$ with $\mm(E) = v$.
\end{proposition}

\begin{proof}
{\bf Step 1.}
Suppose by contradiction the claim was false so that we can find sequences  $\ve_{j} \to 0$, $h_{j}$ of densities verifying $\CD(N-1,N)$
over $[0,\pi - \ve_{j}]$ and $E_{j} \subset [0,D_{j}]$ with $D_{j}: =\pi -\ve_{j}$ such that
$$
\lim_{j\to \infty} \frac{\PP_{h_{j}}(E_{j}) - \I_{h_{j}}(v)}{\min\{ h_{j}\L^{1}( E_{j} \Delta [0,r_{h_{j}}^{-}(v)]), h_{j}\L^{1}( E_{j} \Delta [r_{h_{j}}^{+}(v),\pi])   ) \}} = 0.
$$
Possibly passing to a subsequence, we assume that
$$
\lim_{j\to \infty} \frac{\PP_{h_{j}}(E_{j}) - \I_{h_{j}}(v)}{ h_{j}\L^{1}( E_{j} \Delta [0,r_{h_{j}}^{-}(v)])  } = 0;
$$
it will be clear from the proof that the other case follows similarly.

As $P_{h_{j}}(E_{j})$ is uniformly bounded, we can find a representative of $E_{j}$ (i.e. having the same $\PP_{h}$), that we denote with the same symbol, such that
$$
E_{j} = \bigcup_{i\in \N} (a^{i}_{j},b^{i}_{j}).
$$
By Proposition \ref{P:estimatedensity}, we deduce that $h_{j} \to h_{N}$ uniformly over any compact subset of $(0,\pi)$;   moreover, since $h_{j} \geq 0,  h_{j}^{1/N}$ is concave and  $\int_{0}^{D_{j}} h_{j}=1$ then $\sup_{j} \sup_{t \in [0,D_{j}]} h_{j} < \infty$.  In particular
$\I_{h_{j}}(v) \to \I_{\pi}(v)$
yielding that also $\PP_{h_{j}}(E_{j})$ converges to $\I_{\pi}(v)$, hence we infer by compactness, that
$h_{j} \chi_{E_{j}} \to h_{N} \chi_{[0,r_{N}^{-}(v)]}$ pointwise over $[0,1]$ and in $L^{1}_{loc}(0,\pi)$.
Possibly passing to a subsequence, it follows that $\chi_{E_{j}} \to \chi_{[0,r_{N}^{-}(v)]}$ pointwise (recall that also $r_{h_{j}}^{-}(v) \to r_{N}^{-}(v)$ as $j\to \infty$).
For ease of notation $r_{h_{j}}^{-}(v) =r_{h_{j}}(v)$ and the same for $r_{N}^{-}(v)$.

From Proposition \ref{P:estimatedensity} we deduce that $E_{j}$ can be decomposed as follows
$$
E_{j} = E^{0}_{j} \cup (\beta_{j}, r_{h_{j}}(v)+\gamma_{j})\cup E_{j}^{D_{j}},
$$
where $E^{0}_{j} \subset [0,\eta_{j}], E_{j}^{D_{j}} \subset [D_{j} - \eta_{j}, D_{j}]$ with $\eta_{j} \to 0$ as $j \to \infty$, and $\beta_{j}, \gamma_{j}>0$ (no restrictive) such that
$$
\eta_{j}< \beta_{j} \to 0, \qquad r_{h_{j}}(v)+\gamma_{j} < D_{j} -\eta_{j}, \gamma_{j} \to 0.
$$
\medskip

{\bf Step 2.}  Using again Proposition \ref{P:estimatedensity} ,  the unique maximum $x_{j} \in  [0,D_{j}]$  of $h_{j}$ given by Lemma \ref{L:monotonicity} is necessarily converging to $\pi/2$,
hence if we replace $E_{j}^{0}$ with $[0,\alpha_{j}]$ and $E_{j}^{D_{j}}$ with $[\xi_{j}, D_{j}]$ such that
$$
h_{j}\L^{1} (E_{j}^{0}) = h_{j}\L^{1} ([0,\alpha_{j}]), \qquad h_{j}\L^{1} (E_{j}^{D_{j}}) = h_{j}\L^{1} ([\xi_{j},D_{j}]),
$$
and we call again the new sequence $E_{j}$, then the perimeter will be decreased and the symmetric difference with $[0,r_{h_{j}}^{-}(v)]$  remaining the same.

Hence pick as (possibly) new sequence of sets $E_{j} : = [0,\alpha_{j}] \cup (\beta_{j}, r_{h_{j}}(v)+\gamma_{j})\cup [\xi_{j},D_{j}]$ and from the volume constraint
$$
\int_{(0,\alpha_{j})} h_{j} \,dt+ \int_{(\beta_{j}, r_{h_{j}}(v)+\gamma_{j}))} h_{j} \,dt + \int_{(\xi_{j},D_{j})} h_{j} \,dt = \int_{(0,r_{h_{j}}(v))} h_{j} \,dt
$$
giving
$$
\int_{(r_{h_{j}}(v), r_{h_{j}}(v)+\gamma_{j}))} h_{j} \,dt + \int_{(\xi_{j},D_{j})} h_{j} \,dt = \int_{(\alpha_{j},\beta_{j})} h_{j} \,dt.
$$
This permits to obtain the next identity:
\begin{align}
h_{j}\L^{1}( E_{j} \Delta [0,r_{h_{j}}^{-}(v)])  \nonumber
&~= \int_{(\alpha_{j},\beta_{j})} h_{j}(t)\,dt + \int_{(r_{h_{j}}(v),r_{h_{j}}(v)+\gamma_{j})} h_{j}(t)\,dt  + \int_{(\xi_{j},D_{j})} h_{j}(t)\,dt   \nonumber \\
&~= 2\int_{(\alpha_{j},\beta_{j})} h_{j}(t)\,dt; \label{eq:gamma}
\end{align}
note that by monotonicity of $h_{j}$ (recall that $\beta_{j} \to 0$)
$$
\int_{(\alpha_{j},\beta_{j})} h_{j}(t)\,dt \leq \beta_{j} h_{j}(\beta_{j}).$$
\medskip

{\bf Step 3.}
Define $F_{j} : = [\alpha_{j}, r_{h_{j}}(v) + \gamma_{j}] \cup [\xi_{j},D_{j}]$  and notice that
$$
v_{j}: = \int_{F_{j}} h_{j}(t)\,dt = v + \int_{(\alpha_{j},\beta_{j})} h_{j}(t)\,dt  -  \int_{(0,\alpha_{j})} h_{j}(t)\,dt,
$$
giving by monotonicity $|v_{j} - v| \leq \beta_{j} h_{j}(\beta_{j})$.
Hence expanding $\PP_{h_{j}}(E_{j})$  as
$$\PP_{h_{j}}(E_{j})=\PP_{h_{j}}(F_{j})+ h_{j}(\beta_{j})\geq \cI_{h_{j}}(v_{j})+h_{j}(\beta_{j})$$   one obtains
$$
\PP_{h_{j}}(E_{j}) - \I_{h_{j}}(v)\geq h_{j}(\beta_{j})  + \I_{h_{j}}(v_{j}) -  \I_{h_{j}}(v).
$$
To conclude we observe that the map $(0,1) \ni v \to r_{h_{j}}^{\pm}(v)$ is differentiable with derivative equals to $1/h_{j}(r_{h_{j}}^{\pm}(v))$;
this together with the Lipschitz regularity of $h_{j}$ (with Lipschitz constant uniform on $j$ and depending just on $N$ and  $v$, see Corollary \ref{C:A2}),
implies that
\begin{equation}\label{eq:delta}
|\I_{h_{j}}(v_{j}) -  \I_{h_{j}}(v)| \leq C_{N,v} | v_{j} - v| \leq C_{N,v} \beta_{j} h_{j}(\beta_{j}).
\end{equation}
Then we obtain a contradiction noticing that  the combination of \eqref{eq:gamma} and \eqref{eq:delta} gives
$$
\frac{\PP_{h_{j}}(E_{j}) - \I_{h_{j}}(v) }{h_{j}\L^{1}(E_{j} \Delta [0,r_{h_{j}}(v)]))} \geq
\frac{h_{j}(\beta_{j}) (1 -\beta_{j} C_{N,v} )  }{2 \int_{(\alpha_{j},\beta_{j})} h_{j}(t)\,dt };
$$
in particular, for  $j$ large, the right hand side (recall that $\beta_{j} \to 0$ and $\int_{(\alpha_{j},\beta_{j})} h_{j}(t)\,dt \leq \beta_{j}h_{j}(\beta_{j})$) is arbitrarily large.
The claim follows.
\end{proof}


\subsection{Isoperimetric profile}

We now study the behaviour of the model isoperimetric profile functions in terms of the diameter upper bounds.
For any $D \in [0,\pi]$  there exists $\xi \in [0,\pi - D]$ such that the one-dimensional $\CD(N-1,N)$  model space  with diameter $D$ is given 
by $([\xi, \xi + D], |\cdot|, \sin(t)^{N-1}/(\omega_{N}\lambda_{D}))$ where
\begin{equation}\label{E:volume-profile}
\lambda_{D}: = \frac{1}{\omega_{N}}\int_{\xi}^{\xi +D} \sin(t)^{N-1} \, dt;
\end{equation}
$\omega_{N}$ being the renormalization constant (i.e. the volume of the $N$-dimensional sphere of  unit   radius,   in case $2\leq N \in \N$);
in particular, $\lambda_{\pi} = 1$. For ease of notation we set $h_{D} := \sin(t)^{N-1}/(\omega_{N}\lambda_{D})$ and we omit the dependence on $D$ of $\xi$; moreover for coherence 
$h_{\pi}$ will be denoted with $h_{N}$.

For each volume $v\in (0,1)$,  consider  the two intervals $[\xi, r_{D}^{-}(v)]$, $[r_{D}^{+}(v), D + \xi] \subset [\xi, \xi+D]$ with
$$
\int_{\xi}^{r_{D}^{-}(v)} h_{D}(t) \, dt= \int_{r_{D}^{+}(v)}^{D + \xi} h_{D}(t) \, dt=v.
$$
Then  \cite[Proposition 2.1]{Bobkov} implies that  $\min\{h_{D}(r_{D}^{-}(v)),h_{D}(r_{D}^{+}(v))\}=\cI_{D}(v):=\cI_{N-1,N,D}(v)$, the model isoperimetric profile.

\begin{lemma}\label{L:I_{D}}
The following holds:
\begin{equation}\label{E:I_{D}}
\I_{D}(v) = \frac{1}{\lambda_{D}} \min \left\{ \I_{\pi} \left(\lambda_{D}v + \int_{0}^{\xi} h_{N} dt \right), \I_{\pi}\left(\lambda_{D}(1-v) + \int_{0}^{\xi} h_{N}dt \right) \right\}
\end{equation}
\end{lemma}

\begin{proof}
Observe that
$$
\lambda_{D} v =  \int_{\xi}^{r_{D}^{-}(v)} h_{N}(t) \, dt = \int_{0}^{r_{D}^{-}(v)} h_{N}(t) \, dt  - \int_{0}^{\xi} h_{N}(t) \, dt ,
$$
and that $\lambda_{D} h_{D}(r_{D}^{-}(v))  =  h_{N}(r_{D}^{-}(v)) = \I_{\pi}(\lambda_{D}v + \int_{0}^{\xi} h_{N})$.
Analogously
\begin{align*}
\lambda_{D} v &=  \int_{r_{D}^{+}(v)}^{D + \xi} h_{N}(t) \, dt = \int_{r_{D}^{+}(v)}^{\pi} h_{N}(t) \, dt - \int_{D+\xi}^{\pi}h_{N}(t) \, dt  \\
&=  \int_{r_{D}^{+}(v)}^{\pi} h_{N}(t) \, dt 
- \left(1-\lambda_{D} - \int_{0}^{\xi}h_{N}(t) \, dt\right),
\end{align*}
showing that
$$
\lambda_{D} h_{D}(r_{D}^{+}(v)) = h_{N}(r_{D}^{+}(v)) = \I_{\pi}\left(\lambda_{D}(1- v) + \int_{0}^{\xi}h_{N}(t) \, dt\right).
$$
The claim follows.
\end{proof}

As we are going to consider the deficit, the following inequalities will be the relevant ones.

\begin{lemma}[Concavity of $I_{\pi}$]\label{L:concave-profile}
The following estimate holds:
\begin{align*}
\min & \left\{\I_{\pi}\left(\lambda_{D}v + \int_{0}^{\xi} h_{N} \, dt\right), \I_{\pi}\left(\lambda_{D} (1-v) + \int_{0}^{\xi} h_{N} \, dt\right) \right \} - \lambda_{D}\I_{\pi}(v)  \\
& \geq C_{N,v} \min \{ \lambda_{D}^{ \frac{N-1}{N}}, 1-\lambda_{D} \},
\end{align*}
where $C_{N,v}$ is an explicit constant  depending just on $N$ and $v$:
$$
C_{N,v}:=\min\left\{\cI_{\pi}(v)-v \cI'_{\pi}(v), \,\cI_{\pi}(v) + (1-v) \cI'_{\pi}(v),\,   \lim_{t\downarrow 0} \frac{\cI_{\pi}(t)}{t^{(N-1)/N}} \,  \min\{ v^{\frac{N-1}{N}}, (1-v)^{\frac{N-1}{N}}  \}  \right\}.
$$
\end{lemma}

\begin{proof}
{\bf Step 1.} \\ 
We fist study the asymptotics for $\lambda_{D}\to 1$. 
Note that
\begin{align*}
\I_{\pi}\left(\lambda_{D}v + \int_{0}^{\xi} h_{N}\right) - \lambda_{D} \I_{\pi}(v)
= &~  \I_{\pi}\left(\lambda_{D}v + \int_{0}^{\xi} h_{N}\right) - \I_{\pi}(v)  + (1-\lambda_{D})\I_{\pi}(v) \crcr
= &~  \left( v  (\lambda_{D} - 1)  + \int_{0}^{\xi} h_{N}\right)   \I_{\pi}'(P_{\lambda_{D}}) + (1-\lambda_{D})\I_{\pi}(v),
\end{align*}
for some $P_{\lambda_{D}}$ between $\lambda_{D} v + \int_{0}^{\xi} h_{N}$ and  $v$. 
In particular for $\lambda_{D} \to 1$, necessarily $D \to \pi$ and $\xi \to 0$ hence
$$
\liminf_{\lambda_{D} \to 1} \frac{ \int_{0}^{\xi} h_{N}}{1-\lambda_{D}} = \liminf_{\lambda_{D} \to 1} \frac{ \int_{0}^{\xi} h_{N}}{ \int_{0}^{\xi} h_{N} \, dt+  \int_{D+\xi}^{\pi} h_{N} \, dt}  = c, 
$$
with $0 \leq c \leq 1$. In particular 
$$
 \liminf_{\lambda_{D} \to 1}   \frac{\I_{\pi}\left(\lambda_{D}v + \int_{0}^{\xi} h_{N}\right)  - \lambda_{D} \I_{\pi}(v) }{1-\lambda_{D}} = \I_{\pi}(v) +(  c - v) \I_{\pi}'(v). 
$$
Being a linear function of $c\in [0,1]$, the last quantity is larger than the minimum between $\I_{\pi}(v)  - v \I_{\pi}'(v)$ and $\I_{\pi}(v) +( 1 - v) \I_{\pi}'(v)$.
Since $\I_{\pi}(v) = \I_{\pi}(1-v)$, the strict concavity of $\I_{\pi}$ yields that  both are strictly positive. 
\\ Using again the symmetry   $\I_{\pi}(v) = \I_{\pi}(1-v)$, with the  arguments above one can prove that the same lower bounds hold for $\liminf_{\lambda_{D} \to 1} \frac{\I_{\pi}\left(\lambda_{D}(1-v) + \int_{0}^{\xi} h_{N}\right)  - \lambda_{D} \I_{\pi}(v) }{1-\lambda_{D}}$.

\medskip

{\bf Step 2.}
Again from  the strict convexity of $\I_{\pi}$, for $\lambda_{D}$ sufficiently close to $0$,  we have that 
$\I_{\pi}(\lambda_{D}v + \int_{0}^{\xi} h_{N}) \geq \I_{\pi}(\lambda_{D}v)$ 
yielding
$$
\liminf_{\lambda_{D}\to 0} \frac{\I_{\pi}(\lambda_{D}v + \int_{0}^{\xi} h_{N})}{ (\lambda_{D})^{(N-1)/N}}  \geq   \left(\lim_{t\downarrow 0} \frac{\cI_{\pi}(t)}{t^{(N-1)/N}} \right)  v^{\frac{N-1}{N}}.
$$
Analogously we have 
$
\liminf_{\lambda_{D}\to 0} \frac{\I_{\pi}(\lambda_{D}(1-v) + \int_{0}^{\xi} h_{N})}{ (\lambda_{D})^{(N-1)/N}}  \geq   \left(\lim_{t\downarrow 0} \frac{\cI_{\pi}(t)}{t^{(N-1)/N}} \right)  (1-v)^{\frac{N-1}{N}},
$
completing the proof.
\end{proof}


\section{Reduction to the one dimensional case}\label{s:reductionto1d}

From now on we consider fixed a metric measure space $(X,\sfd,\mm)$ that is essentially non-branching and it verifies $\CD(K,N)$ with $K>0$.
Recall that with no loss in generality, we can assume $\supp[\mm] = X$ and  $K = N-1$, giving $\diam(X) \leq \pi$; moreover $(X,\sfd)$ is compact.

We also fix once for all $E \subset X$ together with the associated localization given by the $L^{1}$-optimal transport problem
between
$$
\mu_{0} : = \frac{1}{v} \, \chi_{E} \cdot \mm, \qquad  \mu_{1} : = \frac{1}{1-v}\, \chi_{E^{c}} \cdot \mm,
$$
with $v = \mm(E)$. Following Section \ref{Ss:localization} and Section \ref{Ss:L1OT},
we fix also a $1$-Lipschitz Kantorovich potential $\f : X \to \R$ (it is actually unique up to a constant) such that if
\begin{equation}\label{E:Gamma}
\Gamma : = \{  (x,y) \in X\times X \colon \f(x) - \f(y) = \sfd(x,y)\},
\end{equation}
then a transport plan $\pi$ is optimal if and only if $\pi(\Gamma) = 1$.
Then from $\f$ one obtains the family of transport rays $\{ X_{q}\}_{q \in Q}$ with $Q$ Borel subset of the transport set $\T$;
it is also immediate (see Theorem \ref{T:localize}) to observe that $\mm(X \setminus \T)=0$,
so we have the following disintegration formula:
$$
\mm = \int_{Q} \mm_{q}\, \qq(dq),  \qquad  \qq \in \mathcal{P}(Q).
$$
Moreover for $\qq$-a.e.  $q \in Q$:
\begin{itemize}
\item[-] $\mm_{q}(X_{q}) = 1$;
\item[-] $(X_{q}, \sfd, \mm_{q})$ is a  $\CD(N-1,N)$ space (see Section \ref{S:back});
\item[-] $\mm_{q}(E) = \mm(E) = v \in (0,1)$.
\end{itemize}
Define the deficit $\delta(E) : = \PP(E) - \I_{\pi}(v)$.

\begin{lemma}
The following inequalities hold true:
\begin{equation}\label{E:starting}
\delta(E) \geq \int_{Q} \PP_{q}(E_{q}) - \I_{\pi}(v) \, \qq(dq) \geq \int_{Q} \I_{D_{q}}(v) - \I_{\pi}(v) \, \qq(dq),
\end{equation}
where $\I_{D} = \I_{N-1,N,D}$, for any $D \in [0,\pi]$ and $D_{q} = |X_{q}| = \sfd(P_{S}(q),P_{N}(q))$.
\end{lemma}

\begin{proof}
As observed in Section \ref{Ss:L1OT}, the map $Q \ni q \mapsto |X_{q}|$ is Borel and therefore the same holds for
$Q \ni q \mapsto \lambda_{q} : = \lambda_{|X_{q}|}$, where  $\lambda_{D}$ has been defined in \eqref{E:volume-profile}.
In particular, the last integral makes sense.

To obtain the measurability of $q \mapsto \PP_{q}(E_{q})$ one can argue as follows. Consider any sequence $u_{n} \to \chi_{E}$
such that
$$
\PP(E) + \ve \geq \lim_{n\to \infty} \int_{Q} |\nabla u_{n}| \mm =  \lim_{n\to \infty} \int_{Q}  \int |\nabla u_{n}| \mm_{q} \, \qq(dq) \geq
\int_{Q} \liminf_{n\to \infty}   \int |\nabla u_{n}| \mm_{q} \, \qq(dq),
$$
where the last inequality follows from Fatou's Lemma. This implies that for $\qq$-a.e. $q \in Q$
$$
\liminf_{n\to \infty}   \int |\nabla u_{n}| \mm_{q} < \infty.
$$
Hence $\PP_{q}(E_{q}) < \infty$ $\qq$-a.e and therefore we can deduce that  $\qq$-a.e. $E_{q}$ has representative given by a
countable union of intervals; in particular
$$
\PP_{q}(E_{q}) = \sum_{i} h_{q}(x_{q,i}),
$$
where $x_{q,i}$ belongs to the boundary of $E_{q}$ in $X_{q}$.  To obtain $\qq$-measurability of $q \mapsto \PP_{q}(E_{q})$ it is then enough
to prove $\qq$-measurability of $\sum_{i} h_{q}(x_{q,i})$.

For this purpose we consider $\tilde Q \subset Q$ Borel such that for each
$q \in \tilde Q$ the set $\{x_{q,i} \}_{i}$ is countable.
Moreover for each $q \in \tilde Q$, points of the boundary of $E_{q}$ in $X_{q}$ can be obtained as follows: for each $n$ consider
$$
\Lambda_{n} : = \{ (q,x,y,z) \in \tilde Q \times \T^{b} \times \T^{b} \times \T^{b} \colon (q,x), (q,y) \in R^{b},\, y \in E, \, z \in E^{c}, \sfd(x,y), \sfd(x,z) \leq 1/n \},
$$
and define $\Lambda : = \cap_{n} P_{1,2} (\Lambda_{n})$, where $P_{1,2}$ stands for the projection on the first and second component; notice that for each $n \in \N$ $P_{1,2} (\Lambda_{n})$ is analytic and therefore $\Lambda$ is analytic as well.
Now we notice that each section $\Lambda(q):=\{x \in \T^{b} \, : \, (q,x)\in \Lambda\}$ is countable; then we can invoke a classical result of Lusin (see \cite[Theorem 5.10.3]{Srivastava})
yielding the existence of countably many $x_{i} : \tilde Q \to \T^{b}$ such that
$$
\Lambda = \cup_{i\in \N} \{ (q,x_{i}(q)) \colon q \in \tilde Q\},
$$
with $\{ (q,x_{i}(q)) \colon q \in \tilde Q\}$, analytic; in particular for each $i \in \N$ the map $\tilde Q\ni q \mapsto x_{i}(q)$ is Borel;
it follows that for each $q \in \tilde Q$, the map $\sum_{i} h_{q}(x_{i}(q)) = \PP_{q}(E_{q})$ is $\qq$-measurable.
Hence also the second integral makes sense. To conclude it is enough to use Fatou's Lemma as before and observe that $|\nabla u_n| \circ X_q \geq |(u_n \circ X_q)'| $ $\qq$-a.e. (here, with a slight abuse of notation, we denote by $X_{q}:[0,D_{q}] \to X$ the ray map).
\end{proof}

\subsection{Long rays}

For any $v \in (0,1)$ clearly $\I_{\pi}(v) = \I_{\pi}(1-v)$, then Lemma \ref{L:concave-profile} applies together with Lemma \ref{L:I_{D}}
yielding from \eqref{E:starting} the following inequality
\begin{align}\label{E:stepzero}
\delta(E) \geq &~ C_{N,v} \int_{Q } \frac{1}{\lambda_{q}} \min \{ \lambda_{q}^{(N-1)/N}, 1-\lambda_{q} \} \,\qq(dq) \crcr
= &~ C_{N,v}\left( \int_{Q_{s}} \lambda_{q}^{-1/N} \qq(dq)  + \int_{Q_{\ell}}  \frac{1-\lambda_{q}}{\lambda_{q}} \, \qq(dq) \right),
\end{align}
 where $s$ and $\ell$ stand for ``short'' and ``long'',  and
\begin{align*}
Q_{s} : = &~\{ q \in Q \colon \lambda_{q}^{(N-1)/N} \leq 1- \lambda_{q} \} = \{ q \in Q \colon \lambda_{q} \leq \eta_{N} \} , \\
Q_{\ell} : = &~\{ q \in Q \colon \lambda_{q}^{(N-1)/N} > 1- \lambda_{q} \}= \{ q \in Q \colon \lambda_{q} > \eta_{N} \},
\end{align*}
where $0 <\eta_{N} < 1$ is the unique solution of $x^{(N-1)/N} =  1-x $. Note that $\eta_{N} \to 1/2$ as $N\to \infty$ and
$Q_{s},Q_{\ell}$ are Borel subsets of $Q$.
We continue noticing
$$
\delta(E ) \geq C_{N,v} \int_{Q_{s}} \eta_{N}^{-1/N} \,\qq(dq) = C_{N,v}\eta_{N}^{-1/N} \qq(Q_{s}),
$$
and in particular
\begin{equation}\label{E:deficit-short}
\mm\left( E \cap \big(\cup_{q\in Q_{s}}X_{q}\big)  \right) \leq \qq(Q_{s}) \leq   \eta_{N}^{1/N} C_{N,v}^{-1} \delta(E);
\end{equation}
meaning that  most of the $\mm$-measure of $E$ must be contained in the set spanned by the family of rays denoted with $Q_{\ell}$.

Since $Q_{\ell} \subset \{q\in Q \colon D_{q} \geq D_{N}\}$ for some $D_{N} \in (0,\pi)$, we obtain the next proposition.

\begin{proposition}\label{P:Q2-2}
The following estimate holds
\begin{equation}\label{E:Q2-2}
\delta(E) \geq C_{N,v} C'_{N,v} \int_{Q_{\ell}} (\pi - D_{q})^{N} \, \qq(dq).
\end{equation}
\end{proposition}

\begin{proof}
From \eqref{E:stepzero}
$$
\delta(E) \geq C_{N,v} \int_{Q_{\ell}} \frac{1-\lambda_{q}}{\lambda_{q}} \,\qq(dq) \geq C_{N,v} \int_{Q_{\ell}} (1-\lambda_{q}) \,\qq(dq).
$$
We now study the behavior of $1-\lambda_{q}$ when $D_{q}$ approaches $\pi$ and $D_{N}$.
Since
$$
1-\lambda_{q} = 1-  \int_{\xi_{q}}^{D_{q}+\xi_{q}}h_{N}(t)\,dt = \int_{0}^{\xi_{q}} h_{N}(t)\,dt +\int_{D_{q}+\xi_{q}}^{\pi} h_{N}(t)\,dt,
$$
it follows that
\begin{equation}\label{eq:alpha}
\lim_{D_{q} \to \pi} \frac{1-\lambda_{q}}{ (\pi -D_{q})^{N}} = C_{N}>0;
\end{equation}
notice indeed that $\xi_{q}^{N} + (\pi - D_{q} - \xi_{q})^{N} \geq 2^{1- N} (\pi - D_{q})^{N}$.

It follows therefore the existence  of  an explicit, strictly positive constant $C'_{N,v}$ such that
$$
1-\lambda_{q} \geq C'_{N,v} (\pi - D_{q})^{N}.
$$
The claim follows.
\end{proof}

We can therefore obtain the first main result of the paper (first claim in Theorem \ref{thm:CD}).

\begin{theorem}\label{T:Q2-3}

There exists at least one $\bar q \in Q$ such that
\begin{equation}\label{E:Q2-3}
(\pi - D_{\bar q})^{N} \leq \delta(E) \frac{1}{C'_{N,v} (C_{N,v}  -\delta(E) \eta_{N}^{1/(N-1)}) } \leq C''_{N,v} \delta(E).
\end{equation}
In particular, there exists a constant $C(N,v)$ depending only on $N > 1$ and $v \in (0,1)$ such that
$$
\pi-\diam(X)\leq C(N,v) \; \delta(E)^{1/N}.
$$
\end{theorem}

\begin{proof}
Just observe that from Proposition \ref{P:Q2-2} there exists at least one $\bar q \in Q$ such that
$$
(\pi - D_{\bar q})^{N} \leq \delta(E) \frac{1}{\qq(Q_{\ell}) C_{N,v}C'_{N,v}}.
$$
Then, from \eqref{E:deficit-short}, we get that
\begin{equation}\label{eq:qqQell}
\qq(Q_{\ell}) = 1 - \qq(Q_{s}) \geq 1 - \delta(E) \frac{\eta_{N}^{1/(N-1)}}{C_{N,v}}.
\end{equation}
The claim follows.
\end{proof}

\noindent
Since $D_{\bar q}$ of Theorem \ref{T:Q2-3} will play a key role,  from now on we consider $\bar q$ fixed and given by Theorem \ref{T:Q2-3}.


\section{Structure of the transport set}\label{S:structure}

So far we have observed that the distance of $|X_{q}|$ from $\pi$ is controlled by the deficit $\delta(E)$ (see Proposition \ref{P:Q2-2}).
In this section we use this information to prove that most of the rays starts close to $P_{S}(\bar q)$ and finishes close to $P_{N}(\bar q)$.
The optimality of $\f$ will be crucial.
We will use the following result (that is of interest in itself) giving a bound on the diameter of the complement of a metric ball.

\begin{proposition}\label{P:antipodal}
Given $N> 1$ there exists $C_{N}>0$ such that the next statement holds.

Let $(X,\sfd, \mm)$ be $\CD(N-1,N)$ space (actually $\MCP(N-1,N)$ would be enough).
Let $x,y,z \in X$ be such that $\sfd(x,y)=\sfd(x,z) \geq D$.
Then $\sfd(y,z)\leq C_{N} (\pi-D)$.   In particular, for every $x_{1}\in X$  there exists $x_{2}=x_{2}(x_{1})$ such that
\begin{equation}\label{eq:BpiepsBepsCD}
X\setminus B_{D}(x_{1}) \subset B_{C_{N} \, (\pi - D)} (x_{2}).
\end{equation}
\end{proposition}

\begin{proof}
{\bf Step 1.}
Without loss of generality we can assume $\mm(X)=1$. Call $r:=\frac{1}{2} \sfd(y,z)$  and let $\mu$ be the $(N-1,N)$-model measure
on $[0,\pi]$, i.e.  $\mu = h_{N} \L^{1}$.
First of all by Bishop-Gromov inequality we have
\begin{equation}\label{eq:BSaux}
\mm(B_{r}(y)) \geq \mu([0,r]), \quad \mm(B_{r}(z)) \geq \mu([0,r]), \quad \mm(B_{D-r}(x)) \geq \mu([0,D-r]).
\end{equation}
Moreover, by construction the sets $B_{D-r}(x), B_{r}(y)$ and $B_{r}(z)$ are pairwise disjoint.
Thus
\begin{align*}
1&= \mm(X) \geq  \mm(B_{D-r}(x)) + \mm(B_{r}(y)) +  \mm(B_{r}(z)) \overset {\eqref{eq:BSaux}}{\geq}   \mu([0,{D-r}]) + 2 \mu([0,{r}])  \\
&= \big( \mu([0,{\pi-r}])- \mu([D-r,\pi-r]) \big)   + \mu([\pi-r, \pi])+ \mu([0,{r}])   \\
&=  \big( \mu([0,{\pi-r}])  + \mu([\pi-r, \pi]) \big)  +  \mu([0,{r}])  - \mu([D-r,\pi-r])  \\
& = \mu([0,\pi])  +  \mu([0,r])   -  \mu([D-r, \pi-r]) = 1 +   \mu([0,r])   -  \mu([D-r, \pi-r]).
\end{align*}
It follows that  $ \mu([0,r]) \leq  \mu([D-r, \pi-r])$.
\medskip

{\bf Step 2.}
We consider two different cases: if $r \leq \pi - D$, then we are done. So we can restrict to the case $ r > \pi - D$.
Then we have
$$
\int_{0}^{r} \sin^{N-1}(t)\,dt \leq \int_{D-r}^{\pi-r} \sin^{N-1}(t)\,dt \leq \int_{D-r}^{\pi-r} (\pi -t )^{N-1}\,dt = \frac{r^{N}}{N} \left( \left(\frac{\pi - D}{r} + 1\right)^{N} - 1 \right).
$$
By convexity
$$
(1 + s)^{N} - 1 \leq s (2^{N}-1),
$$
provided $s \leq 1$, yielding
$$
\frac{N}{r^{N}}\int_{0}^{r} \sin^{N-1}(t)\,dt \leq (2^{N}-1) \frac{\pi -D}{r};
$$
noticing that $\inf_{ r \in (0,\pi)}\int_{0}^{r} \sin^{N-1}(t)dt/ r^{N}$ is strictly positive, gives the claim.
\end{proof}

The $\sfd$-monotonicity of $\Gamma$, i.e. for any $(x_{1},y_{1}),\dots, (x_{n},y_{n}) \in \Gamma$
$$
\sum_{i=1}^{n} \sfd(x_{i},y_{i}) \leq \sum_{i=1}^{n} \sfd(x_{i},y_{i + 1}), \qquad y_{n+1} = y_{1},
$$
for any $n \in \N$,  is crucial to obtain the next step.

\begin{lemma}\label{L:deficit-endpoints}
The following estimate holds true:
\begin{align*}
2^{N-1}&~ \left( \frac{1}{C_{N,v} C'_{N,v}} + C''_{N,v}\right)\delta(E)  \crcr
\geq  &~ \int_{Q_{\ell}}\Big(  \big(\pi  - \sfd(P_{S}(q),P_N(\bar q)) \big) +  \big(\pi - \sfd(P_S(\bar q),P_{N}(q) ) \big)\Big)^{N} \, \qq(dq) .
\end{align*}
\end{lemma}

\begin{proof}
The $\sfd$-monotonicity of the transport set implies that for any ray $X_q$ and any $x, y \in X_q$ with $(x,y) \in \Gamma$ it holds
$$
2\pi - \sfd (x,y)- \sfd(P_S(\bar q), P_N(\bar q))\geq 2\pi - \sfd(x,P_N(\bar q)) - \sfd(y, P_S(\bar q)).
$$
Which we can rewrite as
$$
\pi- \sfd(x,y)  +  \pi-D_{\bar{q}}   \geq  \pi  - \sfd(x,P_N(\bar q))+   \pi - \sfd(y, P_S(\bar q)).
$$
In particular if we take $x = P_{S}(q)$ and $y = P_{N}(q)$,
we deduce that
$$
\pi- D_{q}  +  \pi-D_{\bar{q}}   \geq \pi  - \sfd(P_{S}(q),P_N(\bar q)) +  \pi - \sfd(P_S(\bar q), P_{N}(q)).
$$
Then from Theorem \ref{T:Q2-3} it follows that

\begin{align*}
2^{N-1}\Big(( \pi- D_{q} )^{N} + C''_{N,v}\delta(E) \Big) \geq &~ \left( \pi- D_{q}  + (C''_{N,v}\delta(E))^{1/N}  \right)^{N}\crcr
\geq  &~ \Big( \pi  - \sfd(P_{S}(q),P_N(\bar q)) +  \pi - \sfd(P_S(\bar q),P_{N}(q) ) \Big)^{N}.
\end{align*}
Recalling \eqref{E:Q2-2}, we deduce that
\begin{align*}
2^{N-1} &\left( \frac{1}{C_{N,v} C'_{N,v}} + C''_{N,v}\right)\delta(E) \crcr
\geq &~\int_{Q_{\ell}}2^{N-1}\Big(( \pi- D_{q}  )^{N} + C''_{N,v}\delta(E) \Big) \,\qq(dq) \crcr
\geq &~ \int_{Q_{\ell}}\Big(  \big(\pi  - \sfd(P_{S}(q),P_N(\bar q)) \big) +  \big(\pi - \sfd(P_S(\bar q),P_{N}(q) ) \big)\Big)^{N} \, \qq(dq),
\end{align*}
proving the claim.
\end{proof}

\medskip

 For $\beta \in (0,1)$ to be fixed later,   it is then natural to consider the following sets of rays:
\begin{align}\label{eq:defQl1Ql2}
\begin{split}
Q_{\ell}^{1} &: = \{ q\in Q_{\ell} \colon \sfd(P_{S}(q), P_{N}(\bar q))\leq \pi -   \delta(E)^{\frac{\beta}{N}}  \}, \\
Q_{\ell}^{2} &: = \{ q\in Q_{\ell} \colon \sfd(P_{S}(\bar q), P_{N}(q))\leq \pi -   \delta(E)^{\frac{\beta}{N}}  \};
\end{split}
\end{align}
notice that both $Q_{\ell}^{1}$ and $Q_{\ell}^{2}$ are Borel sets.

\begin{lemma}
The following estimates hold true:
$$
\qq(Q_{\ell}^{1}) \leq C(N,v) \,  \delta(E)^{1-\beta}  ,\qquad \qq(Q_{\ell}^{2}) \leq C(N,v)  \,  \delta(E)^{1-\beta} .
$$
\end{lemma}
\begin{proof}
From Lemma \ref{L:deficit-endpoints} we deduce that
\begin{align*}
2^{N-1}&~ \left( \frac{1}{C_{N,v} C'_{N,v}} + C''_{N,v}\right)\delta(E)  \crcr
\geq  &~ \int_{Q_{\ell}}\Big(  \big(\pi  - \sfd(P_{S}(q),P_N(\bar q)) \big) +  \big(\pi - \sfd(P_S(\bar q),P_{N}(q) ) \big)\Big)^{N} \, \qq(dq) \crcr
\geq  &~  \qq(Q_{\ell}^{i})  \,  \delta(E)^{\beta} ,
\end{align*}
for $i = 1,2$, proving the claim.
\end{proof}

We can therefore restrict our analysis to the following family of rays:
\begin{equation}\label{D:goodQ}
Q_{\ell}^{g} : = Q_{\ell} \setminus (Q_{\ell}^{1} \cup Q_{\ell}^{2}),
\end{equation}
where $g$ stands for $good$. Clearly $Q_{\ell}^{g}$ is Borel and $\qq(Q_{\ell}^{g}) \geq  1 - C(N,v)\,   \delta(E)^{1-\beta} $.

We can now prove that also distances between initial and final points are controlled by the deficit, provided $q$ belongs to the set of ``good'' rays.
Proposition \ref{P:antipodal} will be now used in a crucial way.

\begin{corollary}\label{C:positionSN}
There exists a strictly positive constant $C(N,v)$ only depending on $N$ and $v \in (0,1)$ such that
\begin{align*}
\sfd(P_{S}(q), P_{S}(\bar q)), \sfd(P_{N}(q), P_{N}(\bar q)) \leq C(N,v)\,   \delta(E)^{\frac{\beta}{N}} ,
\end{align*}
for each $q \in Q_{\ell}^{g}$.
\end{corollary}

\begin{proof}
By definition for each $q \in Q_{\ell}^{g}$
$$
 \delta(E)^{\frac{\beta}{N}}  > \pi - \sfd(P_{S}(q), P_{N}(\bar q)), \qquad  \delta(E)^{\frac{\beta}{N}} > \pi - \sfd(P_{S}(\bar q), P_{N}( q)).
$$
Moreover from Theorem \ref{T:Q2-3} we have that $C_{N,v}'' \delta(E)^{\frac{1}{N}} \geq  \pi -\sfd(P_{S}(\bar q), P_{N}(\bar q))$.
Hence Proposition \ref{P:antipodal} implies that
$$
C_{N} C_{N,v}''  \delta(E)^{\frac{\beta}{N}}   \geq \sfd(P_{S}(q), P_{S}(\bar q)), \qquad C_{N} C_{N,v}''   \delta(E)^{\frac{\beta}{N}}  \geq \sfd(P_{N}(q), P_{N}(\bar q)),
$$
proving the claim.
\end{proof}

We summarize all the properties obtained so far for the set of good rays $Q_{\ell}^{g}$:
\begin{align}
& \qq(Q_{\ell}^{g}) \geq 1 -  C(N,v)\,  \delta(E)^{1-\beta} ;  \label{eq:qqQlg}  \\
& \text{for each } q \in Q_{\ell}^{g}:\,  \sfd(P_{S}(q), P_{S}(\bar q)), \sfd(P_{N}(q), P_{N}(\bar q)) \leq C(N,v) \,  \delta(E)^{\frac{\beta}{N}}  ; \\
& \text{for each } q \in Q_{\ell}^{g}:\,  D_{q} = \sfd(P_{S}(q), P_{N}( q)) \geq \pi - C(N,v) \,  \delta(E)^{\frac{\beta}{N}} ;
\end{align}
where $C(N,v)> 0$ is a positive constant depending only on $N > 1$ and $v \in (0,1)$, and $\bar q \in Q$ is the distinguished ray from Theorem \ref{T:Q2-3}.

\section{Quantitative isoperimetric inequality}\label{S:QII}

We then want to separate the rays such
that $E_{q} : = E \cap X_{q}$ has optimal competitor staying in the south pole from the ones having it at the north pole.
 For $\gamma\in (0,1)$ to be chosen later,   we therefore continue considering the following subsets of rays:
\begin{align*}
Q_{\ell}^{S} : = &~ \left\{ q \in Q_{\ell}^{g} \colon \mm_{q}( E_{q} \Delta [0,r^{-}_{q}] ) \leq \,   \delta(E)^{\gamma}   \right\}, \crcr
Q_{\ell}^{N} : = &~ \left\{ q \in Q_{\ell}^{g} \colon \mm_{q}( E_{q} \Delta [r^{+}_{q}, D_{q}] ) \leq   \,   \delta(E)^{\gamma}   \right\},
\end{align*}
where $D_{q} = |X_{q}|$ and $r_{q}^{\pm}:= r_{h_{q}}^{\pm}(v) \in (0,D_{q})$, with $\mm_{q} = h_{q}\mathcal{L}^{1}$,
are the unique points such that
$$
v= \int_{0}^{r_{h_{q}}^{-}(v)} h_{q}(t)\,dt = v = \int_{r_{h_{q}}^{+}(v)}^{D_{q}} h_{q}(t)\,dt.
$$
We will show that at least one of the previous set of rays must have small measure.

First we need to prove the $Q_{\ell}^{S},Q_{\ell}^{N}$ are measurable; we start with the following measurability result

\begin{lemma}
For any $v \in [0,1]$, the maps
$$
Q \ni q  \mapsto r^{\pm}_{h_{q}}(v)
$$
are Borel.
\end{lemma}

\begin{proof}
First we recall  from Section \ref{Ss:L1OT} that the density $(q,t) \mapsto h_{q}(t)$ is Borel.
Then from Fubini's Theorem, for each $r \in \R$,
$$
Q \ni q \mapsto \int_{0}^{r} h_{q}(t)\,dt
$$
is Borel; since for each $q \in Q$, $r \to \int_{0}^{r} h_{q}(t)\,dt$ is continuous, it follows that $Q \times \R \ni (q,r) \to \int_{0}^{r} h_{q}(t)\,dt$ is Borel.
Then
$$
\gr(r^{-}) = \left\{ (q,r) \in Q \times \R \colon v = \int_{0}^{r} h_{q}(t)\,dt \right\},
$$
yields that $\gr(r^{-})$ is Borel and the claim follows. The same holds true for $r^{+}$.
\end{proof}

Then we can conclude as follows that $Q_{\ell}^{S},Q_{\ell}^{N}$ are measurable: from the proof  of the previous Lemma, we get  that for each $v \in (0,1)$
$$
\Lambda : =  g( \{ (q,t) \in Q \times [0,\infty) \colon t \leq r_{q}^{-}(v) \} )
$$
is an analytic set ($\gr(r^{-})$ is Borel);  since
$$
\mm_{q}(E_{q} \Delta [0,r_{q}^{-}(v)] ) =  \mm_{q} (E \Delta \Lambda),
$$
from the measurability of the disintegration, it follows that $Q_{\ell}^{S}$ is $\qq$-measurable. The same holds for $Q_{\ell}^{N}$; possibly passing
to subsets with same $\qq$-measure, we can assume both of them to be Borel.

\medskip

We now pass to analyze $Q_{\ell}^{S}$ and $Q_{\ell}^{N}$.

 We first show the next lemma.

\begin{lemma}\label{L:rayNS}
Using the notation above it holds
\begin{equation}\label{eq:QgsmQNS}
\qq(Q^{g}_{\ell}\setminus (Q^{N}_{\ell} \cup Q^{S}_{\ell})) \leq \frac{1}{C(N,v) \,} \delta^{1-\gamma}(E).
\end{equation}
\end{lemma}

\begin{proof}
From Proposition \ref{P:1dquant} we know that for $\delta(E)$ sufficiently small
$$
\PP_{h_{q}}(E_{q}) - \I_{h_{q}}(v) \geq C(N,v) \min\{ \mm_{q}( E_{q} \Delta [0,r_{q}^{-}(v)]), \mm_{q}( E_{q} \Delta [r_{q}^{+}(v),D_{q}])    \}.
$$
We infer
\begin{align*}
\delta(E) &= \PP(E)-\cI_{\pi}(v) \geq \int_{Q} \left(\PP_{h_{q}}(E_{q}) -\cI_{\pi}(v) \right) \, \qq(dq)  \\
& \geq C(N,v) \, \int_{Q_{\ell}^{g} \setminus (Q^{N}_{\ell} \cup Q^{S}_{\ell})}   \min\{ \mm_{q}( E_{q} \Delta [0,r_{q}^{-}(v)]), \mm_{q}( E_{q} \Delta [r_{q}^{+}(v),D_{q}])    \} \, \qq(dq)  \\
& \geq  C(N,v) \,  \delta(E)^{\gamma} \, \qq( Q_{\ell}^{g} \setminus (Q^{N}_{\ell} \cup Q^{S}_{\ell})),
\end{align*}
giving the claim.
\end{proof}

For reader's convenience we include here an easy one-dimensional result.

\begin{lemma}\label{L:easy}
Let $f : X \to [0,1]$ be a Borel function such that $\int f(x) \xi(dx) = c >0$, with $\xi$ positive finite Borel measure.
Then
$$
\xi \left( \{ x \in X \colon f(x) \geq a \} \right) \geq  \frac{c - aK}{1-a},
$$
where  $K = \xi(X)$.
\end{lemma}

\begin{proof}
Just note that
\begin{align*}
c & = ~ \int_{\{ f \geq a\}} f \xi + \int_{\{ f < a\}} f \xi \leq  \xi(\{ f \geq a\}) + a \xi(\{ f < a\})  \\
&=  \xi(\{ f \geq a\}) + a (K- \xi(\{ f \geq a\})) \\
&=  \xi(\{ f \geq a\})(1 - a) + a K,
\end{align*}
and the claim follows.
\end{proof}

\begin{proposition}\label{P:nointerface}
For any $C >0$,  and any $\alpha >0$ such that
$$
\alpha <   \frac{N}{2N-1} \; \min\{\gamma, 1-\gamma, 1-\beta  \} ,
$$
there exists $\bar \delta > 0$ such that, whenever $\delta(E)< \bar \delta$, then the following inequality holds
$$
\min \{ \qq(Q_{\ell}^{S}) , \qq(Q_{\ell}^{N})\} \leq C \, \delta(E)^{\alpha}.
$$
\end{proposition}

\begin{proof}

Suppose by contradiction the claim was false:
$$
\qq(Q_{\ell}^{S}) , \qq(Q_{\ell}^{N})> C \, \delta(E)^{\alpha},
$$
with $\alpha$ verifying the inequality of the statement; then we argue as follows.

Consider the set
$$
E_{S} : = \bigcup_{q\in Q_{\ell}^{S}} E_{q}, \qquad  E_{N} : = \bigcup_{q\in Q_{\ell}^{N}} E_{q},
\qquad E^{b} : = \bigcup_{q \in Q_{s}\cup Q_{\ell}^{b}} E_{q},
$$
 where $Q_{\ell}^{b}:=Q_{\ell}^{1}\cup Q_{\ell}^{2}=Q_{\ell} \setminus Q_{\ell}^{g}$; notice that $E^{b}$ coincide up to a set of $\mm$-measure zero with $E \setminus E_{S} \cup E_{N}$.
We will accordingly decompose the perimeter of $E$ and eventually find a contradiction for small deficit. \medskip

{\bf Step 1.} \\
Consider a ball $B_{r}(P_{S}(\bar q))$, that for ease of notation we simply denote with   $B^{S}_{r}$, with $r > 0$ such that
\begin{itemize}
\item [-]  for each  $q \in Q_{\ell}^{N}$ the interval $[r_{q}^{+}, D_{q}] \cap B^{S}_{3r} = \emptyset$;
\item [-] for each  $q \in Q_{\ell}^{S}$,  $\sfd(P_{S}( q), \partial B^{S}_{3r} ) < r_{q}^{-}-\ve$, for some $\ve>0$
\end{itemize}
where $B^{S}_{3r}$ denotes the ball centered  in $P_{S}(\bar q)$ as well  with radius $3r$.
\\For $q \in Q_{\ell}^{g}$ it holds $D_{q} \geq \pi - C(N,v) \,  \delta(E)^{\frac{\beta}{N}}$ , implying (see Proposition \ref{P:estimatedensity}) that
 $$
| r_{q}^{-} - r_{N}^{-}(v) |, | r_{q}^{+} - r_{N}^{+}(v) | \leq C(N,v) \,  \delta(E)^{\frac{\beta}{N}}
$$
showing that we can chose $r$ sufficiently small so the the previous properties are verified, at least for $\delta(E)$ sufficiently small.
Notice that as $\delta(E)$ approaches $0$, $r$ can be considered fixed.

We now estimate the amount of mass of $E$ contained in $B^{S}_{r}$:
\begin{align*}
\mm(E\cap B^{S}_{r})& \geq  \mm(E_{S}\cap B^{S}_{r}) = \int_{Q^{S}_{\ell}} \mm_{q}(E_{q} \cap B^{S}_{r}) \qq(dq)  \\
&\geq \int_{Q_{\ell}^{S}} \mm_{q}( [0,r_{q}^{-}] \cap B^{S}_{r}) - \mm_{q}( E_{q} \Delta [0,r_{q}^{-}] )   \,\qq(dq).
\end{align*}
By triangular inequality $[0,r_{q}^{-}] \cap B^{S}_{r} \supset [0, r - \sfd(P_{S}(\bar q), P_{S}(q))] \supset [0, r- C(N,v)   \delta(E)^{\frac{\beta}{N}} ]$;
then we can continue as follows
\begin{align}\label{E:EBS}
\geq &~ C(r)\qq(Q_{\ell}^{S}) -\,   \delta(E)^{\gamma}  \geq  C(r)\delta(E)^{\alpha} - \,   \delta(E)^{\gamma} ,
\end{align}
with $C(r)$ only depending on the radius of $B^{S}_{r}$ remaining positive
when $\delta(E)$ approaches $0$ (see again Proposition \ref{P:estimatedensity}).

Moreover since $B^{S}_{r} \cap [r_{q}^{+}, D_{q}] = \emptyset$, it follows that $B^{S}_{r} \setminus E =  (B^{S}_{r}\setminus (E \setminus [r_{q}^{+}, D_{q}])),$
and therefore
\begin{align}\label{E:BSE}
\mm(B^{S}_{r} \setminus E) \geq &~ \int_{Q_{\ell}^{N}} \mm_{q}(B^{S}_{r} \setminus (E \setminus [r_{q}^{+}, D_{q}])) \, \qq(dq)  \geq  \int_{Q_{\ell}^{N}} \mm_{q}(B^{S}_{r}) \, \qq(dq) -\,  \delta(E)^{\gamma}  \cr
\geq &~  C(r)\delta(E)^{\alpha} -  \delta(E)^{\gamma} .
\end{align}

 We will find two contributions to the perimeter of $E$: one coming from the relative perimeter of $E$ inside  $B^{S}_{3r}$  and one coming from the relative perimeter of $E$ inside  $ X \setminus \overline{B^{S}_{3r+\ve}}$. In other words we  decompose (see Section \ref{Ss:isoperimetric})
$$
\PP(E) \geq \PP(E, B^{S}_{3r}) + \PP(E, X \setminus \overline{B^{S}_{3r+\ve}}).
$$
The second contribution will be obtained in Step 3 using the localization of $E$ discussed above; for the first one instead, since we  have not a disposal any isoperimetric inequality inside $B^{S}_{r}$
(that possibly is not a convex subset of $X$),  we will  consider a new localization whose associated transport set is contained in $B^{S}_{3r}$. This will be discussed in the next Step 2.
\\

{\bf Step 2.}
Consider the localization of the following function
$$
f = \frac{\chi_{E\cap B^{S}_{r}}}{\mm(E \cap B^{S}_{r})} - \frac{\chi_{B^{S}_{r} \setminus E}}{\mm(B^{S}_{r} \setminus E)}.
$$
Denote with $\T^{1} $ the corresponding transport set and consider the associated disintegration
$$
\mm\llcorner_{\T^{1}} = \int_{Q^{1}} \mm_{q}^{1} \, \qq^{1}(dq),
$$
verifying for $\qq^{1}$-a.e. $q \in Q^{1}$ the following properties
\begin{itemize}
\item $\mm^{1}_{q}(X_{q}^{1}) = 1$,
\item $(X_{q}^{1}, \sfd, \mm_{q}^{1})$ verifies $\CD(N-1,N)$
\item $\int f \mm_{q}^{1} = 0$.
\end{itemize}
Following moreover \cite{CM4}, since the ray $X^{1}_{q}$ starts inside $E\cap B^{S}_{r}$ and arrives inside $B^{S}_{r}\setminus E$,
one can modify the definition of the transport set and obtain that $X^{1}_{q} \subset B^{S}_{3r}$, at the price of obtaining a decomposition
of a strict subset of the original transport set which however will be still denoted with $\T^{1}$ and still contains $B^{S}_{r}$, up to an $\mm$-negligible subset; for details see \cite[Section 3]{CM4}.
\medskip

By definition of $\PP(E, B^{S}_{3r})$, for some $\{ u_{n}\}_{n\in\N} \subset \Lip(B^{S}_{3r})$ with $u_{n} \to \chi_{E}$ in $L^{1}(B^{S}_{3r},\mm)$
\begin{align*}
\PP(E, B^{S}_{3r})
&~ = \lim_{n\to \infty} \int_{B^{S}_{3r}} |\nabla u_{n}|(x) \, \mm(dx)  \geq \lim_{n\to \infty} \int_{B^{S}_{3r} \cap \T^{1}} |\nabla u_{n}|(x)\,  \mm(dx) \\
&~ \geq \lim_{n\to \infty}  \int_{Q^{1}} \int_{B^{S}_{3r} \cap \T^{1}} |\nabla u_{n}|(x) \,  \mm_{q}^{1}(dx) \qq^{1}(dq).
\end{align*}
Notice now that for $\qq^{1}$-a.e. $q \in Q^{1}$, the map $u_{n}$ restricted to the ray $X_{q}^{1}$ is still Lipschitz and converges to $\chi_{E \cap X_{q}^{1}}$ in
$L^{1}(B^{S}_{3r} \cap X^{1}_{q},\mm^{1}_{q})$; as observed before, $X_{q}^{1} \subset B^{S}_{3r}$,
hence the previous chain of inequalities can be continued using Fatou's Lemma as follows
\begin{equation}\label{eq:bullet}
\PP(E, B^{S}_{3r})  \geq  \int_{Q^{1}} \PP_{q}(E, X^{1}_{q})\, \qq^{1}(dq) \geq \int_{Q^{1}} \I_{N-1,N, \pi}(\mm_{q}^{1}(E)) \, \qq^{1}(dq),
\end{equation}
where $\PP_{q}(E, X^{1}_{q})$ is  the one-dimensional  perimeter of $E$ in the one-dimensional open set $X_{q}^{1}$ with respect to the
one-dimensional measure $\mm_{q}^{1}$ and the last inequality holds thanks to the fact that $(X^{1}_{q}, \sfd, \mm^{1}_{q})$ is a  $\CD(N-1,N)$ space.

Now from the localization, it follows that for $\qq^{1}$-a.e. $q \in Q^{1}$
\begin{equation}\label{E:newlocal}
\mm^{1}_{q}(E\cap B^{S}_{r}) = \frac{\mm(E\cap B^{S}_{r})}{\mm(B^{S}_{r} \setminus E)} \mm^{1}_{q}(B^{S}_{r} \setminus E).
\end{equation}
As $\mm^{1}_{q}$ is a probability measure $\qq^{1}$-a.e. and
$$
\mm(E\cap B^{S}_{r}) = \int_{Q^{1}} \mm^{1}_{q}(E\cap B^{S}_{r}) \, \qq^{1}(dq),
$$
from Lemma \ref{L:easy} we deduce the next inequality, for any $a \in [0,1)$
$$
\qq^{1}(\{ q \in Q^{1} \colon \mm^{1}_{q}(E\cap B^{S}_{r}) \geq a \}) \geq \frac{\mm(E\cap B^{S}_{r})  - a \mm(\T^{1})  }{1-a} \geq
\frac{\mm(E\cap B^{S}_{r})  - a}{1-a}.
$$
Choosing $a = \mm(E\cap B^{S}_{r})/2$  and denoting
$$
\bar Q^{1} : = \{ q \in Q^{1} \colon \mm^{1}_{q}(E\cap B^{S}_{r}) \geq  \mm(E\cap B^{S}_{r})/2 \},
$$
we obtain the next inequality
\begin{equation}\label{E:Q1bar}
\qq^{1}(\bar Q^{1}) \geq \mm(E\cap B^{S}_{r})/2.
\end{equation}
From  the definition of $\bar{Q}^{1}$ and \eqref{E:newlocal} it follows that
$$
\frac{1}{2}\mm(E\cap B^{S}_{r}) \leq \mm^{1}_{q}(E \cap B^{S}_{r}), \quad \frac{1}{2} \mm(B^{S}_{r}\setminus E) \leq \mm^{1}_{q}( B^{S}_{r} \setminus E), \quad \qq^{1}\text{-a.e.} \; q \in \bar{Q}^{1}.
$$
Hence we obtain immediately that
$$
\mm^{1}_{q}(E) \geq \mm^{1}_{q}(E \cap B^{S}_{r}) \geq \frac{1}{2}\mm(E\cap B^{S}_{r}), \quad \qq^{1}\text{-a.e.} \; q \in \bar{Q}^{1},
$$
and
$$
\mm^{1}_{q}(E) \leq 1 - \mm^{1}_{q}(B^{S}_{r} \setminus E) \leq 1 - \frac{1}{2} \mm(B^{S}_{r}\setminus E), \quad \qq^{1}\text{-a.e.} \; q \in \bar{Q}^{1}.
$$
Combining the last estimate with  \eqref{E:EBS} and \eqref{E:BSE} we obtain that
$$
\frac{1}{2}\left(C(r)\delta(E)^{\alpha} -  \,\delta(E)^{\gamma}  \right) \leq \mm^{1}_{q}(E ) \leq
1 - \frac{1}{2}\left(C(r)\delta(E)^{\alpha} - \,   \delta(E)^{\gamma}  \right), \quad \qq^{1}\text{-a.e.} \; q \in \bar{Q}^{1}.
$$
 Assuming $\gamma\in (\alpha,1)$,  this implies that for $\qq$-a.e. $q \in \bar Q^{1}$, $\I_{\pi}(\mm^{1}_{q}(E)) \geq \hat C(N,v,r) \delta(E)^{\alpha\frac{N-1}{N}}$;
hence, recalling \eqref{eq:bullet},  we obtain
\begin{equation}\label{eq:(1)}
\PP(E, B^{S}_{3r}) \geq \hat C(N,v,r) \, \delta(E)^{\alpha\frac{N-1}{N}} \qq(\bar Q^{1}) \geq \hat C(N,v,r) \, \delta(E)^{\alpha\frac{2N-1}{N}} \geq \hat C(N,v) \, \delta(E)^{\alpha \frac{2N-1}{N}},
\end{equation}
where the second inequality follows from \eqref{E:EBS} and \eqref{E:Q1bar}.
\\

{\bf Step 3.}\\
Now we take into account the contribution to the perimeter of $E$ inside $X \setminus \overline{B^{S}_{3r+\ve}}$:
reasoning as at the beginning of ${\bf Step 2.}$, we use the one dimensional rays of the localization of $E$ to obtain the next inequality
\begin{align*}
\PP(E, X \setminus \overline{B^{S}_{3r+\ve}})    \geq &~  \int_{Q_{\ell}} \PP_{q}(E_{q} , X_{q} \setminus \overline{B^{S}_{3r+\ve}})\,\qq(dq)   \crcr
= &~  \int_{Q_{\ell}^{N}} \PP_{q}(E_{q} , X_{q} \setminus \overline{B^{S}_{3r+\ve}} ) \,\qq(dq) +\int_{Q_{\ell}^{S}} \PP_{q}(E_{q} , X_{q} \setminus \overline{B^{S}_{3r+\ve}}) \,\qq(dq),
\end{align*}
where $\PP_{q}(E_{q} , X_{q} \setminus \overline{B^{S}_{3r+\ve}} )$ is  the one-dimensional relative perimeter of $E$ in the one-dimensional open set $X_{q} \setminus \overline{B^{S}_{3r+\ve}} \subset X_{q}$ with respect to the
one-dimensional measure $\mm_{q}$.

For $q \in Q_{\ell}^{S}$, by definition we know that $
\mm_{q}(E_{q} \Delta [0,r_{q}^{-}]) \leq \, \delta(E)^{\gamma} $; in  particular
\begin{equation}\label{eq:(star)}
\mm_{q} (E \setminus [0,r_{q}^{-}]) \leq\,   \delta(E)^{\gamma}.
\end{equation}
Notice that since $\PP_{q}(E_{q}) < \infty$, up to a set of $\mm_{q}$-measure zero, we can assume it to be the countable union of closed sets.
This will not affect any of the quantities involved in this proof. We now claim that
\begin{equation}\label{E:Q2S-point}
[r_{q}^{-}    -  2 \delta(E)^{\gamma}  /C_{N,q,v}   , r_{q}^{-} +  2  \delta(E)^{\gamma}  /C_{N,q,v}] \cap \partial E_{q} \neq \emptyset
\end{equation}
where
$$
C_{N,q,v} : = \min\{ h_{q}(t) \colon t\in  [r^{-}_{q}/2, r^{-}_{q} + (D_{q}-r^{-}_{q})/2] \},
$$
is uniformly positive for $q \in Q_{\ell}$  and $\delta(E) \in (0, \bar{\delta}(N)]$, by Proposition \ref{P:estimatedensity}.
We start the proof of \eqref{E:Q2S-point} by showing that
\begin{equation}\label{E:Q2S-point-bis}
[r_{q}^{-}   -  2 \delta(E)^{\gamma}  /C_{N,q,v}  , r_{q}^{-} ] \cap E_{q} \neq \emptyset.
\end{equation}
So suppose by contradiction that \eqref{E:Q2S-point-bis} was false. Since
$$
\mm_{q} \big( [r_{q}^{-}   -  2 \delta(E)^{\gamma}  /C_{N,q,v}  , r_{q}^{-}] \big) \geq   2 \delta(E)^{\gamma},
$$
we deduce that
\begin{align*}
\mm_{q}(E_{q} \cap [0,r_{q}^{-}]) = 	&~ \mm_{q}(E_{q} \cap [0,r_{q}^{-}  - 2 \delta(E)^{\gamma} ] )  \leq  \mm([0,r_{q}^{-}    - 2 \delta(E)^{\gamma} ]) \crcr
						 \leq 	&~ v - \mm([r_{q}^{-}   - 2 \delta(E)^{\gamma} ,r_{q}]) \crcr
						 \leq 	&~ v    - 2 \delta(E)^{\gamma}.
\end{align*}
It follows that
$$
v = \mm_{q}(E_{q} \setminus [0,r_{q}^{-}]) + \mm_{q}(E_{q} \cap [0,r_{q}^{-}]) \leq \mm_{q}(E_{q} \setminus [0,r_{q}^{-}]) +v   - 2 \delta(E)^{\gamma} ,
$$
 contradicting \eqref{eq:(star)}.
\\Hence \eqref{E:Q2S-point-bis} is proved. To obtain \eqref{E:Q2S-point} observe analogously that
$$
\mm_{q}([r_{q}^{-}, r_{q}^{-} + 2 \,  \delta(E)^{\gamma}  /C_{N,q,v}]) \geq 2\,  \delta(E)^{\gamma} ;
$$
therefore again by \eqref{eq:(star)} we get that $[r_{q}^{-}, r_{q}^{-} + 2  \delta(E)^{\gamma}  / C_{N,q,v}] \setminus  E \neq \emptyset$ yielding the claim \eqref{E:Q2S-point}.
\medskip

From  \eqref{E:Q2S-point}, we deduce that for $q \in Q_{\ell}^{S}$ it holds
$$
\PP_{q}(E_{q} ; X \setminus B^{S}_{3r+\ve}) \geq h_{q}(x_{q}),
$$
with $x_{q} \in [r_{q}^{-}  - 2 \delta(E)^{\gamma} /C_{N,q,v} , r_{q}^{-} +  2  \delta(E)^{\gamma}  /C_{N,q,v}] \cap \partial E_{q}$.
Hence
\begin{align*}
\PP_{q}(E_{q} ;X \setminus B^{S}_{3r+\ve}) ) - I_{\pi}(v)
\geq &~  h_{q}(r_{q}^{-}) - I_{\pi}(v) + h_{q}(x_{q}) - h_{q}(r_{q}^{-})  \crcr
\geq &~  h_{q}(r_{q}^{-}) - I_{\pi}(v) - C   \delta(E)^{\gamma}  \crcr
\geq &~  - C   \delta(E)^{\gamma} ,
\end{align*}
with $C = \sup\{ h_{q}'(t) \colon t\in  [r_{q}^{-}  - 2 \delta(E)^{\gamma} /C_{N,q,v} , r_{q}^{-} +  2  \delta(E)^{\gamma}  /C_{N,q,v}] \}$, uniform in $q \in Q_{\ell}$.
A similar (easier) argument  also works for $q \in Q_{\ell}^{N}$.
\\

{\bf Step 4.} \\
We now collect all the steps to reach a contradiction as follows:
\begin{align*}
\delta(E) 	\geq &~ \PP(E) - \I_{\pi}(v) \crcr
		\geq &~ \PP(E; B^{S}_{3r}) + \PP(E; X \setminus B^{S}_{3r+\ve}) - \I_{\pi}(v)\crcr
		\geq &~ \hat C(N,v) \delta(E)^{\alpha \frac{2N-1}{N}} +
				\int_{Q_{\ell}^{N}} \PP_{q}(E_{q} ; X \setminus B^{S}_{3r+\ve} ) \,\qq(dq) 
				   +\int_{Q_{\ell}^{S}} \PP_{q}(E_{q} ; X \setminus B^{S}_{3r+\ve}) \,\qq(dq)  \\
			& \quad	  -  \I_{\pi}(v) \crcr
		\geq	&~ \hat C(N,v) \delta(E)^{\alpha \frac{2N-1}{N}} + \qq(Q_{\ell}^{N}\cup Q_{\ell}^{S})(\I_{\pi}(v) -C   \delta(E)^{\gamma} ) -  \I_{\pi}(v) \crcr
		\geq	&~ \hat C(N,v) \, \delta(E)^{\alpha \frac{2N-1}{N}} +  \left(1 -  C(N,v) \,  \delta(E)^{1-\beta}  -  \frac{1}{C(N,v) \,} \delta(E)^{1-\gamma}  \right) (\I_{\pi}(v) -C \delta(E)^{\gamma}) \\
		& \quad -  \I_{\pi}(v),
\end{align*}
 where in the last estimate we made use of  \eqref{eq:qqQlg} and  \eqref{eq:QgsmQNS}.
Since all the constants are stable for $\delta(E)$ approaching $0$, the last inequality shows a
contradiction provided
$$
\alpha \frac{2N - 1}{N} <    \min\{\gamma, 1-\gamma, 1-\beta  \},
$$
for $\delta(E)$ below a threshold depending only on $N$ and $v = \mm(E)$.
\end{proof}

\begin{remark}\label{R:alexandrov}
 In case $(X,\sfd)$ is the metric space associated to a smooth Riemannian manifold, then $r$ can be chosen small enough so that  $B^{S}_{r}$ is a convex. In this case  it follows
that  $(B^{S}_{r},\sfd|_{B^{S}_{r}}, \mm|_{B^{S}_{r}})$ is a  non-branching  $\CD(N-1,N)$ space. Therefore Step 2 above  can be simplified as we can directly apply the Levy-Gromov inequality stated in Theorem  \ref{theorem:LGM} and get
the better estimate
$$
\PP(E, B) \geq \hat C(N,v) \delta(E)^{\alpha \frac{N-1}{N}}.
$$
Repeating verbatim the other steps of the proof, we reach a contradiction provided $\alpha$ verifies the  less restrictive
inequality
$$\alpha  <  \frac{N}{N-1} \min\{\gamma, 1-\gamma, 1-\beta  \} .$$
\end{remark}

In particular we have the next result.

\begin{proposition}\label{P:nointerface-Alexandrov}
Suppose $(X,\sfd,\mm)$ is a $\CD(N-1,N)$ space, $N\geq 2$,  with $(X,\sfd)$ metric space associated to a smooth Riemannian manifold.
Then there exists $\bar \delta > 0$ such that, whenever $\delta(E)< \bar \delta$, then the following inequality holds
$$
\min \{ \qq(Q_{\ell}^{S}) , \qq(Q_{\ell}^{N})\} \leq C \, \delta(E)^{\alpha}, \quad \text{for any } \alpha  <  \frac{N}{N-1} \min\{\gamma, 1-\gamma, 1-\beta  \} .
$$

\end{proposition}

We are now in position of proving the other main result of the paper.

\medskip

\textbf{Proof of Theorem \ref{thm:CD}}
We first observe that the first claim has already been obtained in Theorem \ref{T:Q2-3}.
We are therefore left with showing the second claim.

From Proposition \ref{P:nointerface} we assume with no loss in generality that $\qq(Q_{\ell}^{N}) \leq C \,  \delta(E)^{\alpha} $
and consequently  we pick as $\bar x = P_{S}(\bar q)$.
\\We will use now the following notation
$$
Q_{bad} = Q \setminus Q_{\ell}^{S}.
$$
Since from  \eqref{eq:qqQlg} and  \eqref{eq:QgsmQNS} we have
\begin{align}
\qq(Q_{\ell}^{S}) &\geq \qq(Q_{\ell}^{g})- \qq(Q^{g}_{\ell}\setminus (Q_{\ell}^{N}\cup Q_{\ell}^{S}))  - \qq(Q_{\ell}^{N}) \nonumber \\
& \geq  1 -  C(N,v) \,  \delta(E)^{1-\beta}  -  \frac{1}{C(N,v) \,} \delta(E)^{1-\gamma} -  C\delta(E)^{\alpha}  \label{eq:qqQellS}.
\end{align}
  Recalling that $\alpha< \min\{1-\gamma , 1-\beta\}$, we get
$
\qq(Q_{bad} ) \leq  C_{fin} \delta(E)^{\alpha}
$
with $C_{fin}$ depending on $N$ and $v$, and therefore
\begin{align*}
\mm(X \setminus \QQ^{-1}(Q_{bad }))\leq C_{fin} \,   \delta(E)^{\alpha} .
\end{align*}

Using the definition of $Q_{\ell}^{S}$ we obtain
\begin{align*}
\mm\Big( (E \Delta B_{r_{N}^{-}(v)}(\bar x)) \cap \QQ^{-1}(Q_{\ell}^{S}) \Big)
= &~ \int_{Q_{\ell}^{S}} \mm_{q}(E \Delta B_{r_{N}^{-}(v)}(\bar x)) \, \qq(dq) \crcr
\leq &~ \int_{Q_{\ell}^{S}} \mm_{q}(E_{q} \Delta [0,r_{q}^{-}]) \, \qq(dq) + \int_{Q_{\ell}^{S}} \mm_{q}([0,r_{q}^{-}] \Delta B_{r_{N}^{-}(v)}(\bar x)) \, \qq(dq) \crcr
\leq &\;    \delta(E)^{\gamma}  + \int_{Q_{\ell}^{S}} \mm_{q}([0,r_{q}^{-}] \Delta B_{r_{N}^{-}(v)}(\bar x)) \, \qq(dq).
\end{align*}
From Corollary \ref{C:positionSN} we have that $\sfd(P_{S}(q),P_{S}(\bar q)) \leq C(N,v)\,   \delta(E)^{\frac{\beta}{N}} $; it follows that
$$
[[0,r_{q}^{-}] \Delta B_{r_{N}^{-}}(\bar x)] \subset [r_{q}^{-} - C(N,v)   \delta(E)^{\frac{\beta}{N}} , r_{q}^{-} + C(N,v)   \delta(E)^{\frac{\beta}{N}} ]
$$
implying that
$$
\mm_{q}([0,r_{q}^{-}] \Delta B_{r_{N}^{-}}(\bar x))  \leq 2\,  \|h_{q}\|_{\infty}  \, C(N,v) \,  \delta(E)^{\frac{\beta}{N}}.
$$
Since the rays of $Q_{\ell}$ are uniformly long once the deficit is assumed to  be smaller than 1/10, we have $\|h_{q}\|_{\infty} \leq C(N)$ and therefore
$$
\int_{Q_{\ell}^{S}}\mm_{q}([0,r_{q}^{-}] \Delta B_{r_{N}^{-}(v)}(\bar{x})) \,\qq(dq) \leq  2 C(N) \, C(N,v)  \,  \delta(E)^{\frac{\beta}{N}}.
$$
 We  conclude that
\begin{align*}
\mm(E\Delta B_{r_{N}^{-}(v)}(\bar{x})) & \leq  \mm\Big( (E \Delta B_{r_{N}^{-}(v)}(\bar x)) \cap \QQ^{-1}(Q_{\ell}^{S}) \Big) +  \mm(X \setminus \QQ^{-1}(Q_{bad })) \\
& \leq  2 C(N) \, C(N,v)  \,  \delta(E)^{\frac{\beta}{N}} +    C_{fin} \,   \delta(E)^{\alpha} ,
\end{align*}
giving the claim for $\CD(N-1, N)$-spaces for
$$\eta:=\min\left\{   \frac{N}{2N-1} \; \min\{\gamma, 1-\gamma, 1-\beta  \} ,  \frac{\beta}{N}  \right\}.$$
It is easy to check that for $\beta \in (0,1), \gamma\in (0,1)$, the right hand side is maximized for $\gamma=1/2$ and $\beta=\frac{N^{2}}{N^{2}+2N-1}$ giving $\eta=\frac{N}{N^{2}+2N-1}$.

For smooth Riemannian manifolds, we can follow the same arguments by  using Proposition \ref{P:nointerface-Alexandrov} to improving the claim with
$$\eta:=\min\left\{   \frac{N}{N-1} \; \min\{\gamma, 1-\gamma, 1-\beta  \} ,  \frac{\beta}{N}  \right\}.$$
It is easy to check that for $\beta \in (0,1), \gamma\in (0,1)$, the right hand side is maximized for $\gamma=1/2$ and $\beta=\frac{N^{2}}{N^{2}+N-1}$ giving $\eta=\frac{N}{N^2+N-1}$.

\hfill$\Box$


\renewcommand{\thesection}{A}
\section{Appendix - Technical Lemmas} \label{Appendix A}

Let $h : [0,D] \to [0,\infty)$ be a $\CD(N-1,N)$ density, i.e.  for any $0 \leq t_{0}< t_{1} \leq D$  we have
\begin{equation}\label{eq:(alpha)}
\sin(  (t_{1}-t_{0})) h^{\frac{1}{N-1}}((1-s)t_{0} + st_{1})
\geq  \sin\left((1-s)(t_{1}-t_{0})  \right)h^{\frac{1}{N-1}}(t_{0})   + \sin\left(s(t_{1}-t_{0}) \right)h^{\frac{1}{N-1}}(t_{1}).
\end{equation}
Note that the above inequality is nothing but the synthetic version of the differential inequality

$$
\left(h^{1/(N-1)} \right)''+ h^{1/(N-1)} \leq 0;
$$
the previous condition together with $\int_{0}^{D} h(t)\,dt = 1$ implies that $h > 0$ over $[0,D]$.

Here we are interested in studing the behaviour of a general $\CD(N-1,N)$ density $h$ with almost maximal domain in the following sense $\pi - D = \ve \ll 1$.
In particular, we look for uniform estimates, i.e. only depending on $\ve$.

We recall the definition of the model $\CD(N-1,N)$ density:
\begin{eqnarray}
h_{N} (t)&: =& \frac{1}{\omega_{N}}\sin^{N-1}(t ),\quad   t \in [0,\pi],  \nonumber \\
\omega_{N} &: =& \int_{0}^{\pi} \sin ^{N-1}( t ) \, dt.  \nonumber
\end{eqnarray}
Then the following estimates hold.
\begin{lemma}\label{L:CD1estimates}
For any $t \in (0,D)$ the following holds: if $s>0$ is such that $t+ s \leq D$, then
\begin{equation}\label{eq:(ast)}
\frac{h_{N}(t+s+\ve)}{h_{N}(t+\ve)}
\leq \frac{h(t+ s )}{h(t)}
\leq  \frac{h_{N}(t+s)}{h_{N}(t)}.
\end{equation}
\end{lemma}

\begin{proof}
{\bf Step 1.}\\
Assume $t_{1} = D = \pi - \ve$ and $t_{0}$ any element of $(0,D)$:
$$
\sin \left( \pi -  \ve - t_{0} \right) h^{\frac{1}{N-1}}\left(t_{0} +  s (\pi - \ve - t_{0} )  \right)
\geq  \sin\left((1-s)(\pi - \ve -t_{0})  \right) h^{\frac{1}{N-1}}(t_{0})
$$
Calling   $\tau_{0} : = s (\pi - \ve - t_{0})$  then we have
$$
 \frac{h^{\frac{1}{N-1}}\left(t_{0} +  \tau_{0}  \right)}{h^{\frac{1}{N-1}}(t_{0})}
\geq  \frac{\sin\left((1-s)(\pi - \ve -t_{0})\right)}{\sin \left(\pi -  \ve - t_{0} \right) } = \frac{\sin(\pi - \ve - t_{0} - \tau_{0})}{\sin(\pi - \ve - t_{0})}
= \frac{\sin(t_{0}+\tau_{0} + \ve)}{\sin(t_{0}+\ve)},
$$
for every $\tau_{0}>0$ such that $t_{0}+\tau_{0} \leq D$, yielding the first inequality in  \eqref{eq:(ast)}.
\medskip

{\bf Step 2.}\\
Now fix $t_{0} = 0$ and any $t_{1}$  any element of $(0,D)$:
$$
\sin \left( t_{1} \right) h^{\frac{1}{N-1}}\left(  s t_{1}  \right)
\geq  \sin\left( s t_{1}  \right) h^{\frac{1}{N-1}}(t_{1})
$$
Now define $t = st_{1}$ and $\tau: = (1-s) t_{1}$ and obtain:
$$
 \frac{h^{\frac{1}{N-1}}\left(t +  \tau  \right)}{h^{\frac{1}{N-1}}(t)}
\leq \frac{\sin(t +\tau)}{\sin(t)},
$$
for any $t, \tau > 0$ such that $t + \tau  \leq  D$.
 Hence also the second inequality in \eqref{eq:(ast)} is proved.
\end{proof}

Applying logarithm to \eqref{eq:(ast)} and taking the limit as $s \to 0$, one obtains the next bound on the derivative of a general $h$.
Before stating it recall that we set $\lambda_{D}:=\int_{0}^{D} h_{N}(t) \, dt$.

\begin{corollary}\label{C:A2}
Any $\CD(N-1,N)$ density $h:[0,D] \to [0,\infty)$ is locally Lipschitz and
for any $t \in (0,D)$ point of differentiability of $h$
$$
\frac{h'_{N}(t+\ve)}{h_{N}(t+\ve)} \leq \frac{h'(t)}{h(t)} \leq \frac{h'_{N}(t)}{h_{N}(t)},
$$
where $\ve = \pi - D$. In particular, for any $t \in (0,D)$ there exists $\ve_{0} = \ve_{0}(t)$ such that
$$
(\Lip \, h )(t) \leq C(N,t),
$$
for any $\ve \leq \ve_{0}(t)$.
\end{corollary}

From Lemma \ref{L:CD1estimates} is fairly easy to obtain the following result.

\begin{proposition}\label{P:estimatedensity}
For any $t \in (0,D)$
$$
\left( \frac{\omega_{N}}{\omega_{N}\lambda_{D} + \ve } \right) \min \{ h_{N}(t), h_{N}(t+\ve)\}  \leq h(t) \leq \left( \frac{\omega_{N} }{\omega_{N}-\ve} \right) \max\{ h_{N}(t), h_{N}(t+\ve) \}.
$$
with $C = C(N,D)$.
\end{proposition}

\begin{proof}
Fix  $t \in (0,D)$, then from Lemma \ref{L:CD1estimates}
$$
h(t) h_{N}(t+s )\geq  h(t+s ) h_{N}(t )
$$
for any $s >0$ such that $t+s < D$; in particular
\begin{equation}\label{eq:bullet1}
h(t) \int_{t}^{D} h_{N}(r)\,dt \geq h_{N}(t) \int_{t}^{D} h(r)\,dt.
\end{equation}
Similarly, for any $s \leq 0 $ with $t+s > 0$
$$
h(t) h_{N}(t+s+\ve) \geq h(t+s) h_{N}(t+\ve),
$$
implying
\begin{equation}\label{eq:bullet2}
h(t) \int_{\ve}^{t+\ve} h_{N}(r)\,dt \geq h_{N}(t+\ve) \int_{0}^{t} h(r)\,dr.
\end{equation}
Adding \eqref{eq:bullet1} and \eqref{eq:bullet2} one obtains
$$
\left(\int_{\ve}^{t+\ve} h_{N}(r)\,dt + \int_{t}^{D} h_{N}(r)\,dt \right)  h(t) \geq \min\{ h_{N}(t), h_{N}(t+\ve) \}.
$$
Since $\int_{\ve}^{t+\ve} h_{N}(r)\,dt + \int_{t}^{D} h_{N}(r)\,dt \leq \lambda_{D} + \ve/\omega_{N}$, it follows that
$$
\left(1+\frac{\ve}{\omega_{N}\lambda_{D}}\right)h(t) \geq \frac{1}{\lambda_{D}} \min\{ h_{N}(t), h_{N}(t+\ve) \},
$$
implying the first part of the claim.

The second part follows analogously: for $s > 0$
$$
h(t) h_{N}(t+s+\ve) \leq h(t+s) h_{N}(t+\ve),
$$
implying
$$
h(t) \int_{t+\ve}^{\pi}h_{N}(r)\,dt \leq h_{N}(t+\ve) \int_{t}^{D} h(r)\,dr.
$$
For any $s\leq 0$ with $t+s > 0$:
$$
h(t) h_{N}(t+s) \leq h(t+s) h_{N}(t),
$$
yielding $h(t) \int_{0}^{t} h_{N}(r)\,dr \leq h_{N}(t) \int_{0}^{t}h(r)\,dr$; summing the two contributions one obtains
$$
h(t)\left(1- \int_{t}^{t+\ve}h_{N}(r)\,dr\right) \leq \max\{ h_{N}(t),h_{N}(t+\ve) \};
$$
since $\int_{t}^{t+\ve}h_{N}(r)\,dr \leq \ve/\omega_{N}$, we proved the claim.
\end{proof}

We will also use the following easy monotonicity property.

\begin{lemma}\label{L:monotonicity}
Any density $h : [0,D] \to (0,\infty)$ verifying $\CD(N-1,N)$ that integrates to $1$ has a unique maximum $x_{0} \in  [0,D] $.
Moreover, $h$ is strictly increasing on $[0,x_{0}]$ and strictly decreasing over $[x_{0},D]$.
\end{lemma}

\begin{proof}
By definition, for any $t_{0},t_{1}$ it holds
$$
h^{\frac{1}{N-1}}((1-s) t_{0} + s t_{1}) \geq \sigma_{N-1,N}^{(1-s)}(t_{1} -t_{0})h^{\frac{1}{N-1}}(t_{0} ) +
\sigma_{N-1,N}^{(s)}(t_{1} -t_{0})h^{\frac{1}{N-1}}(t_{1} );
$$
since
$$
\sigma_{N-1,N}^{(s)}(t_{1} -t_{0}) = \frac{\sin(s(t_{1} - t_{0})) }{\sin(t_{1} - t_{0})} > s,
$$
in particular $h^{\frac{1}{N-1}}$ is strictly concave, implying the claim.
\end{proof}


\renewcommand{\thesection}{B}
\section{Appendix - Disintegration Theorem} \label{Appendix B}

Given a measure space $(X,\mathscr{X},\mm)$, suppose
a \emph{partition} of $X$ is given into \emph{disjoint} sets $\{ X_{\alpha}\}_{\alpha \in Q}$ so that $X = \cup_{\alpha \in Q} X_\alpha$.
Here $Q$ is the set of indices and $\QQ : X \to Q$ is the quotient map, i.e.
$$
\alpha = \QQ(x) \iff x \in X_{\alpha}.
$$
We endow $Q$ with the \emph{push forward $\sigma$-algebra} $\mathscr{Q}$ of $\mathscr{X}$:
$$
C \in \mathscr{Q} \quad \Longleftrightarrow \quad \QQ^{-1}(C) \in \mathscr{X},
$$
i.e. the biggest $\sigma$-algebra on $Q$ such that $\QQ$ is measurable; and
Moreover a measure $\qq$ on $(Q,\mathscr{Q})$ can be obtained by pushing forward $\mm$ via $\QQ$, i.e. $\qq := \QQ_\sharp \, \mm$,
obtaining  the quotient measure space $(Q, \mathscr{Q}, \qq)$.

\begin{definition}[Consistent and Strongly Consistent Disintegration]
\label{defi:dis}
A \emph{disintegration} of $\mm$ \emph{consistent with the partition} is a map:
$$
Q \ni \alpha \longmapsto \mm_{\alpha} \in \mathcal{P}(X,\mathscr{X})
$$
such that the following requirements hold:
\begin{enumerate}
\item  for all $B \in \mathscr{X}$, the map $\alpha \mapsto \mm_{\alpha}(B)$ is $\qq$-measurable;
\item for all $B \in \mathscr{X}$ and $C \in \mathscr{Q}$, the following consistency condition holds:
$$
\mm \left(B \cap \QQ^{-1}(C) \right) = \int_{C} \mm_{\alpha}(B)\, \qq(d\alpha).
$$
\end{enumerate}
A disintegration of $\mm$ is called \emph{strongly consistent} if in addition:
\begin{enumerate}
\item[(3)] for $\qq$-a.e. $\alpha \in Q$, $\mm_\alpha$ is concentrated on $X_{\alpha} = \QQ^{-1}(\alpha)$;
\end{enumerate}
\end{definition}

\medskip

We now formulate the Disintegration Theorem (it is formulated for probability measures but clearly holds for any finite non-zero measure):

\begin{theorem}[Theorem A.7, Proposition A.9 of \cite{biacar:cmono}] \label{T:disintegrationgeneral}
Assume that $(X,\mathscr{X},\mm)$ is a countably generated probability space and that $\{X_{\alpha}\}_{\alpha \in Q}$ is a partition of $X$.
\medskip

Then the quotient probability space $(Q, \mathscr{Q},\qq)$ is essentially countably generated and
there exists an essentially unique disintegration $\alpha \mapsto \mm_{\alpha}$ consistent with the partition.
\medskip

If in addition $\mathscr{X}$ contains all singletons, then the disintegration is strongly consistent if and only if there exists a $\mm$-section $S_{\mm} \in \mathscr{X}$ of the partition such that the $\sigma$-algebra on $S_{\mm}$ induced by the quotient-map contains $\mathcal{B}(S_{\mm})$.
\end{theorem}

Let us expand on the statement of Theorem \ref{T:disintegrationgeneral}. Recall that a $\sigma$-algebra $\mathcal{A}$ is \emph{countably generated} if there exists a countable family of sets so that $\mathcal{A}$ coincides with the smallest $\sigma$-algebra containing them.
In the measure space $(Q, \mathscr{Q},\qq)$, the $\sigma$-algebra $\mathscr{Q}$ is called \emph{essentially countably generated} if there exists a countable family of sets $Q_{n} \subset Q$ such that for any $C \in \mathscr{Q}$ there exists $\hat C \in \hat{\mathscr{Q}}$,
where $\hat{\mathscr{Q}}$ is the $\sigma$-algebra generated by $\{ Q_{n} \}_{n \in \N}$, such that $\qq(C\, \Delta \, \hat C) = 0$.
Moreover from \cite[Proposition 3.3.2]{Srivastava} every countably generated measurable space having singletons as atoms is isomorphic to a subset of
$([0,1], \mathcal{B}([0,1]))$; in particular, there exists a topology over $X$ such that $\mathscr{X}$ coincide with the Borel $\sigma$-algebra of the topology;
so the notation $\mathcal{B}(S_{\mm})$ is justified.

Essential uniqueness is understood above in the following sense: if $\alpha\mapsto \mm^{1}_{\alpha}$ and $\alpha\mapsto \mm^{2}_{\alpha}$
are two consistent disintegrations with the partition then $\mm^{1}_{\alpha}=\mm^{2}_{\alpha}$ for $\qq$-a.e. $\alpha \in Q$.

Finally, a set $S \subset X$ is a section for the partition $X = \cup_{\alpha \in Q}X_{\alpha}$ if for any $\alpha \in Q$, $S \cap X_\alpha$ is a singleton $\{x_\alpha\}$.
By the axiom of choice, a section $S$ always exists, and we may identify $Q$ with $S$ via the map $Q \ni \alpha \mapsto x_\alpha \in S$.
A set $S_{\mm}$ is an $\mm$-section if there exists $Y \in \mathscr{X}$ with $\mm(X \setminus Y) = 0$ such that the partition $Y = \cup_{\alpha\in Q_\mm} (X_{\alpha} \cap Y)$ has section $S_{\mm}$, where $Q_{\mm} = \{\alpha \in Q ; X_{\alpha} \cap Y \neq \emptyset\}$. As $\qq = \QQ_{\sharp} \mm$, clearly $\qq(Q \setminus Q_{\mm}) = 0$.
As usual, we identify between $Q_{\mm}$ and $S_{\mm}$, so that now $Q_{\mm}$ carries two measurable structures: $\mathscr{Q} \cap Q_{\mm}$ (the push-forward of $\mathscr{X} \cap Y$ via $\QQ$), and also $\mathscr{X} \cap S_{\mm}$ via our identification. The last condition of Theorem \ref{T:disintegrationgeneral} is that $\mathscr{Q} \cap Q_{\mm} \supset \mathscr{X} \cap S_{\mm}$, i.e. that the restricted quotient-map $\QQ|_{Y} : (Y,\mathscr{X} \cap Y) \rightarrow (S_\mm , \mathscr{X} \cap S_{\mm})$ is measurable, so that the full quotient-map $\QQ : (X,\mathscr{X}) \rightarrow (S , \mathscr{X} \cap S)$ is $\mm$-measurable.

\medskip

We will typically apply the Disintegration Theorem to $(E,\mathcal{B}(E),\mm\llcorner_{E})$, where $E \subset X$ is an $\mm$-measurable subset
(with $\mm(E) > 0$) of the m.m.s. $(X,\sfd,\mm)$.
As our metric space is separable, $\mathcal{B}(E)$ is countably generated, and so Theorem \ref{T:disintegrationgeneral} applies.

\end{document}